\documentclass{article}
\usepackage{amsmath}
\usepackage[left=3cm, right=3cm, bottom=2.8cm, top=2.8cm]{geometry}
\usepackage{hyperref}
\usepackage{cleveref}
\usepackage{amssymb}
\usepackage{mathtools}
\usepackage{tikz-cd}
\usepackage{amsthm}
\usepackage{graphicx}
\usepackage{graphics}
\usepackage{mathrsfs}
\usepackage{cleveref}
\usepackage{thmtools}
\usepackage{enumitem}
\usepackage{bbm}
\usepackage{dsfont}
\usepackage[toc,page]{appendix}

\newlist{steps}{enumerate}{1}
\usepackage{mathtools}
\makeatletter
\newcommand{\xMapsto}[2][]{\ext@arrow 0599{\Mapstofill@}{#1}{#2}}
\def\Mapstofill@{\arrowfill@{\Mapstochar\Relbar}\Relbar\Rightarrow}
\makeatother
\usepackage{array}
\setlist[steps, 1]{label = Step \arabic*:}
\theoremstyle{plain}
\usepackage{footnote}
\makesavenoteenv{tabular}
\newtheorem*{theorem*}{Theorem}
\newtheorem{theorem}{Theorem}[subsection]
\newtheorem{definition}[theorem]{Definition}

\newtheorem{proposition}[theorem]{Proposition}
\newtheorem{remark}[theorem]{Remark}
\newtheorem{corollary}[theorem]{Corollary}
\newtheorem*{corollary*}{Corollary}
\newtheorem*{proposition*}{Proposition}
\newtheorem{definition*}{Definition}
\newtheorem{exmp}[theorem]{Example}

\newtheorem{lemma}[theorem]{Lemma}
\newtheorem*{lemma*}{Lemma}

\numberwithin{equation}{subsection}
\usepackage{rotating}
\usepackage{comment}
\usepackage{mathabx,epsfig}
\def\acts{\mathrel{\reflectbox{$\righttoleftarrow$}}}
\usepackage[nottoc,notlot,notlof]{tocbibind} 

\title{On an Axiomatization of Path Integral Quantization and its Equivalence to Berezin's Quantization}
\author{Joshua Lackman\footnote{josh@pku.edu.cn}}
\date{}
\begin{document}
\maketitle
\begin{abstract}
\noindent 
We axiomatize path integral quantization of symplectic manifolds. We prove that this path integral formulation of quantization is equivalent to an abstract operator formulation, ie. abstract coherent state (or Berezin) quantization. We use the corresponding path integral of Poisson manifolds to quantize all complete Riemann surfaces of constant non–positive curvature and some Poisson structures on the sphere.
\end{abstract}
\maketitle
\tableofcontents
\section{Introduction}
Let $(\mathcal{L},\nabla,\langle\cdot,\cdot\rangle)\to (M,\omega/\hbar)$ be a prequantum line bundle with Hermitian connection and let $P(\gamma)$ denote parallel transport between the fibers of $\mathcal{L}$ over the curve $\gamma.$ Consider the coherent state path integral, which initially appeared in \cite{klauder0}, \cite{klauder} and most recently was discussed in \cite{brane}:
\begin{equation}\label{paths}
\int_{\gamma(0)=x}^{\gamma(1)=y} P(\gamma)\,\mathcal{D}\gamma\;\in \textup{Hom}(\mathcal{L}_x,\mathcal{L}_y)\;.
\end{equation}
We give an axiomatic definition of this path integral and show that it is equivalent to an abstract operator formalism, ie. coherent state (or Berezin) quantization (\cite{ber1}). We define the integrand as a formal inverse limit of an inverse system of complex measures on the space of paths. We axiomatize the corresponding path integral for Poisson manifolds as well. 
\\\\As we show, defining such a path integral is equivalent to Berezin's quantization (\cite{poland}), which is a very strong form of quantization. A brief explanation is as follows: $\mathcal{L}\to M$ is determined by a classifying map 
\begin{equation}
\begin{tikzcd}
              & \mathcal{H}\backslash\{0\} \;\;\;\;\;\; \;\;\;\;\;\;\;   \arrow[d,"\pi",shift right=8] \\
q:M \arrow[r] & \textup{P}(\mathcal{H})\subset \textup{B}(\mathcal{H})                
\end{tikzcd}
    \end{equation}
for a separable Hilbert space $\mathcal{H},$ where we are identifying points in $\textup{P}(\mathcal{H})$ with rank–one orthogonal projections in $\textup{B}(\mathcal{H}).$ Roughly, we show that computing \ref{paths} determines a classifying map $q$ with the overcompleteness property
\begin{equation}\label{id}
    \mathds{1}_{\mathcal{H}}=\int_M q(x)\,\frac{\omega_x^n}{\hbar^{n}}\;.\footnote{More accurately, the integration is with respect to a measure which is equal to $\omega^n/\hbar^n+\mathcal{O}(1/\hbar^{n-1}).$}
\end{equation}
In particular, this identifies $\mathcal{H}$ as the physical Hilbert space. In the other direction, it is a simple observation that given such a $q,$ \ref{paths} is the canonical projection map $\mathcal{L}_x\to\mathcal{L}_y$ of the associated line bundle, which a physicist would write as $|x\rangle \to |y\rangle\langle y|x\rangle.$
\\\\As a brief review of the rest of Berezin's quantization, from \cref{id} it follows that there is an identity–preserving quantization map which restricts to a map into the Hilbert–Schmidt operators, given by
\begin{equation}
   Q\vert_{L^2}:L^2(M)\to \textup{B}(\mathcal{H})_{\textup{HS}}\;,\;\; Q_f:=\int_M f(x)q(x)\,\frac{\omega^n_x}{\hbar^n}\;.
\end{equation}
This is Berezin's contravariant symbol. In the physics literature, this would be  written as
\begin{equation}
    Q_f=\int_M f(z)|z\rangle\langle z|\,\frac{\omega^n_z}{\hbar^n}\;.
\end{equation}
Assuming $q(M)$ has enough states, the adjoint map (Berezin's covariant symbol)
\begin{equation}
    Q^{\dagger}:\textup{B}(\mathcal{H})_{\textup{HS}}\to L^2(M)\cap C^{\infty}(M)
\end{equation}
is injective and thus a dense subspace of $(\ker{Q\vert_{L^2}})^{\perp}$ inherits a noncommutative product.
\\\\
For symplectic manifolds, these ideas have been implemented by choosing a compatible almost complex structure — these have been used to quantize a very large class of prequantizable symplectic manifolds, including all compact ones (\cite{bor},\cite{kor}). By contrast, no such result can exist for the Kostant–Souriau prescription. On the other hand, a general description of a quantization map for Poisson manifolds has been elusive. The corresponding path integral for Poisson manifolds still has the completeness property, but $q$ doesn't necessarily map into $\textup{P}(\mathcal{H}),$ so it can quantize some non–prequantizable symplectic manifolds as well, eg. all complete Riemann surfaces of constant non–positive curvature.
\\\\
In conventional quantum mechanics, $q(x,p)$ is the projection map onto the eigenstate of the lowering operator $\hat{x}+i\hat{p}$ for which $\langle\hat{x}\rangle =x,\,\langle\hat{p}\rangle=p.$ This is a coherent state and it minimizes the uncertainty principle. The corresponding noncommutative product is the non-formal Wick algebra, ie. the normal–ordered product.
\\\\For the Berezin–Toeplitz quantization of compact K\"{a}her manifolds (\cite{bord}, \cite{cahen}, \cite{mar}), $q$ is given by a normalized Kodaira embedding. In the case of $S^2\xhookrightarrow{}\mathbb{R}^3,$ the image of $Q^{\dagger}$ is the space of polynomials in the coordinate functions of degree at most $\textup{deg}(\mathcal{L})$. 
Another source of examples of such quantizations comes from irreducible unitary representations of Lie groups (\cite{klauder4}). This construction is dual to Kirillov's orbit method.
\\\\We introduce a category of abstract coherent state quantizations — this is an abstraction of the quantization obtained by a map $q$ satisfying \cref{id}. We prove it is equivalent to the category of path integrals, and we compute a path integral for:
\begin{enumerate}
    \item All complete Riemann surfaces of constant curvature $\le 0,$ including quotient stacks of the form $[\mathbb{H}/\Gamma],$ $\Gamma\subset \textup{PSL}(2,\mathbb{R}).$
    \item A Poisson structure with a quartic zero at the north pole of $S^2.$
    \item The Podl\`{e}s sphere, which is an $SU(2)$–invariant Poisson structure on $S^2$ with a quadratic zero at the north pole.
    \end{enumerate}
A $C^*$–algebra for the Podl\`{e}s sphere was described in \cite{pod}, but a quantization map is absent. The algebra we use for a Poisson manifold $(M,\Pi)$ is a subalgebra of sections of the prequantum line bundle over the symplectic groupoid, and in nice cases, this subalgebra embeds as a vector space into $C^{\infty}(M).$ This work is in the same vein as the Poisson sigma model and the symplectic groupoid approach to quantization, \cite{bon}, \cite{catt}, \cite{eli}, \cite{weinstein}, \cite{weinstein1}.
\\\\ Some coherent state path integrals were rigorously constructed  using Brownian motion in \cite{klauder}, \cite{klauder2}, by Daubechies and Klauder.\footnote{For related work concerning time evolution, see \cite{charles}.} We take a different approach — since the usual practice of defining a path integral depends on the approximation scheme and is thus ill–defined, we instead define the entire category of path integrals. Essentially, this category contains the limits of all finite dimensional approximation schemes to \cref{paths}.  

\section{Abstract Coherent State Quantization}
We introduce an abstract definition of coherent state quantization for symplectic manifolds; this is an abstraction of the quantization scheme described by Berezin (\cite{ber1}, \cite{sawin}). We will show that the category of path integral quantizations is equivalent to category of abstract coherent state quantizations. We can get more general quantizations by relaxing the projection axiom, eg. some non–prequantizable symplectic manifolds.
\\\\We use $W^*$–algebras. Concretely, these are  weakly–closed $^*$–subalgebras of $\textup{B}(\mathcal{H}),$ ie. von Neumann algebras, whereas $C^*$–algebras are only norm–closed. More details follow.
\begin{definition}\label{maind}
An abstract coherent state quantization of a connected manifold with a Borel measure $(M,d\mu)$ is given by a continuous injection into a $W^*$–algebra
\begin{equation}
q:M\xhookrightarrow{} M_{\hbar}
\end{equation}
such that:
\begin{enumerate}
    \item (projection axiom) $q(x)$ is a minimal projection for all $x\in M,$ ie.
\begin{equation}
    q(x)\ne 0\;,\;\;q(x)^2=q(x)=q(x)^*\;,\;\;q(x)M_{\hbar}q(x)=\mathbb{C}q(x)\;.
    \end{equation}
    \item (overcompleteness axiom) In the weak sense,
    \begin{equation}
        \mathds{1}=\int_M q(x)\,d\mu(x)\;.
    \end{equation} 
     \item[3*.](separation axiom) We say that an abstract coherent state quantization has enough states if $q(x)Aq(x)=0$ for all $x\in M$ implies that $A=0.$
     \end{enumerate}
     \end{definition}
\begin{exmp}\label{proto}
The simplest and prototypical example of such a quantization is given by the canonical inclusion $\mathbb{C}\textup{P}^n\xhookrightarrow{q}\textup{B}(\mathbb{C}^{n+1}),$ taking a point in $\mathbb{C}\textup{P}^n$ to the orthogonal projection onto its corresponding subspace in $\mathbb{C}^{n+1}.$ Here, 
\begin{equation}\label{measf}
  d\mu=\frac{(n+1)}{\pi^n}\omega^n_{\textup{FS}}\;.  
\end{equation}
That the overcompleteness axiom holds can be checked directly, 
and that the separation axiom holds follows from the basic fact that on a complex Hilbert space, $\langle v,Av\rangle=0$ for all $v$ implies that $A=0.$ 
\\\\Every $f\in L^{2}(\mathbb{C}\textup{P}^n)$ determines an operator $Q_f\in \textup{B}(\mathbb{C}^{n+1}),$ given by
\begin{equation}
    Q_f=\frac{(n+1)}{\pi^n}\int_{\mathbb{C}\textup{P}^n}fq\,\omega^n_{\textup{FS}}\;.
\end{equation}
The adjoint to $Q$ with respect to the Hilbert–Schmidt inner product is given by $A\mapsto Q_A^{\dagger},$ where
\begin{equation}
    Q^{\dagger}_A([x])=\langle x,Ax\rangle
\end{equation}
for any normalized $x\in [x]\in \mathbb{C}\textup{P}^n.$ Since $Q^{\dagger}$ is injective its image inherits a noncommutative product, which in this case is such that (\cite{rok})
\begin{equation}
  [Q^{\dagger}_A,Q^{\dagger}_B]=i\{Q^{\dagger}_A,Q^{\dagger}_B\}\;. 
\end{equation}
It's worth emphasizing that this means that $Q^{\dagger}_{[A,B]}=i\{Q^{\dagger}_A,Q^{\dagger}_B\},$ and in particular, the quantum and classical equations of motion are equivalent.
\end{exmp}
\subsubsection{Basic Theory of $W^*$–Algebras}
We describe the basic theory of $W^*$–Algebras, otherwise known as von Neumann algebras. We use these rather than $C^*$–algebras because the completeness axiom is more natural in this setting. In what follows, we assume that $\mathcal{H}$ is a separable Hilbert space (which may be finite dimensional).
\begin{definition}
A $W^*$–algebra $\mathcal{A}$ is a $C^*$–algebra that admits a predual $\mathcal{A}_*,$ ie. $\mathcal{A}_*$ is a Banach space and $\mathcal{A}\cong\mathcal{A}_*^*.$ The topology on $\mathcal{A}$ is the weak$^*$–topology, called the ultraweak topology.
\end{definition}
Note that, the ultraweak topology is independent of the particular choice of predual.
\begin{exmp}
For a separable Hilbert space $\mathcal{H},$ $\textup{B}(\mathcal{H})$ is a $W^*$–algebra whose predual is the trace–class operators.
\end{exmp}
\begin{theorem}
Every $W^*$–algebra faithfully embeds as a weakly closed subspace of $\textup{B}(\mathcal{H}),$ for some Hilbert space $\mathcal{H}.$
\end{theorem}
Therefore, when discussing a $W^*$–algebra $\mathcal{A},$ without loss of generality we may assume $\mathcal{A}\subset \textup{B}(\mathcal{H}).$ However, the topology is a bit different than might be expected.
\begin{definition}
Let $\mathcal{A}\subset \textup{B}(\mathcal{H})$ be a weakly closed subspace. The ultraweak topology is the topology such that $A_n\xrightarrow[]{n\to\infty}A$ if for all $x_i\in\mathcal{H},\,i=1,\ldots,$ such that $\sum_{i=1}^{\infty}\|x_i\|^2<\infty,$
\begin{equation}
    \sum_{i=1}^{\infty}\langle x_i,A_nx_i\rangle\xrightarrow[]{n\to\infty} \sum_{i=1}^{\infty}\langle x_i,Ax_i\rangle\;.
\end{equation}
\end{definition}
Note that, if $A_n\to A$ in the ultraweak topology then $A_n\to A$ in the weak topology.
\\\\The following result will be used to show that the concrete definition of coherent state (or Berezin) quantization is equivalent to the abstract one. It is theorem 4.2.1 in \cite{jones}, and it says that type 1 factors are isomorphic to $B(\mathcal{H})$:
\begin{lemma}\label{cen}
If a $W^*$–algebra $\mathcal{A}$ has a non–zero minimal projection and its center contains only multiples of the identity, then $A\cong \textup{B}(\mathcal{H})$ for some Hilbert space $\mathcal{H}.$
\end{lemma}
\subsection{Basic Theory of Abstract Coherent State Quantization}
Here we describe the basic theory of abstract coherent state quantization. As we will see, without loss of generatlity one can assume that $M_{\hbar}=\textup{B}(\mathcal{H})$ for some Hilbert space $\mathcal{H}.$ Given this, most results of this part and \cref{adjs} are standard (\cite{poland}) and we will go through them quickly. In \cref{3point} we discuss the cohomology class associated to such a quantization.
\\\\First, we describe the meaning of the overcompleteness axiom. Let $M_{\hbar*}\subset M_{\hbar}^*$ be the predual of $M_{\hbar}.$ The overcompleteness axiom is equivalent to: for all $s\in M_{\hbar*},$ the map $x\mapsto s(q_x)$ is in $L^1(M,d\mu)$ and
\begin{equation}
    s(\mathds{1})=\int_M s(q_x)\,d\mu(x)\;.
\end{equation}
\begin{definition}
We have a continuous, $^*$-linear map $\rho:M_{\hbar}\times M\to \mathbb{C}$ defined by
\begin{equation}
   q_xAq_x=\rho_A(x)q_x\;.
\end{equation}
\end{definition}
In particular, $\rho_{\mathds{1}}(x)=1$ or all $x.$
\begin{lemma}\label{important}
There exists a Hilbert space $\mathcal{H}$ such that $M_{\hbar}\cong \textup{B}(\mathcal{H}).$
\end{lemma}
\begin{proof}
Suppose that $A\in M_{\hbar}$ is in the center. By \cref{cen}, it is enough to show that $A$ is a multiple of the identity. We have that
\begin{equation}
    Aq_x=Aq_x^2=q_xAq_x=\rho_A(x) q_x\;.
\end{equation}
This implies that for any $x,y\in M$ and $B\in M_{\hbar},$
\begin{equation}
  \rho_A(x)q_xBq_y=\rho_A(y)q_xBq_y\;.
\end{equation}
Therefore, if $q_xBq_y\ne 0$ it follows that $\rho_A(x)=\rho_A(y).$ Fix $x\in M.$ We claim that 
\begin{equation}\label{set}
    \{y\in M:\textup{there exists }B\in M_{\hbar}\textup{ such that }q_xBq_y\ne 0\}=M\;.
\end{equation}
This set is nonempty because it contains $x.$ It is open because $q$ is continuous and if $q_xBq_y\ne 0$ then there exists an open neighborhood $U\ni q_y$ such that for all $C\in U,$ $q_xBC\ne 0.$ It is closed because its complement is open: suppose $q_y$ is such that $q_xBq_y=0$ for all $B,$ and let $U\ni y$ be an open neighborhood such that for all $z\in U,$ $q_yq_z\ne 0.$ This implies that $\rho_{q_y}(z)\ne 0.$ By assumption, it must be true that 
\begin{equation}
  q_x(Bq_z)q_y=0\;.  
\end{equation}
Therefore,
\begin{equation}
\rho_{q_y}(z)q_xBq_z=q_x(Bq_z)q_yq_z=0\;,
\end{equation}
from which we deduce that $q_xBq_z=0$ for all $B$ as well. This completes the proof of \cref{set}, which implies that $x\mapsto \rho_A(x)$ is constant. By the overcompleteness axiom it follows that $A=\rho_{A}\mathds{1},$ and this completes the proof.
\end{proof}
\begin{proposition}\label{trace}
Let $\mathcal{H}$ be such that $\textup{B}(\mathcal{H})\cong M_{\hbar}.$ Then $\textup{dim}\,\mathcal{H}=\textup{Vol}_{\mu}(M).$
\end{proposition}
\begin{proof}
Using the overcompleteness axiom and taking the trace, it follows that
\begin{equation}
\textup{Tr}(\mathds{1}_{\mathcal{H}})=\int_M \,d\mu=\textup{Vol}_{\mu}(M)\;.
\end{equation}
\end{proof}
\begin{proposition}
Let $s\in M_{\hbar*}$ and let
\begin{equation}
   f\in \bigcup_{\{1\le p\le \infty\}}L^p(M,d\mu)\;.
\end{equation}
Then $x\mapsto f(x)s(q_x)$ is in $L^1(M,d\mu).$ 
\end{proposition}
\begin{proof}
Let $1/p+1/q=1$ and write $|f(x)s(q_x)|=|f(x)||s(q_x)|^{1/p}|s(q_x)|^{1/q}.$ H\"{o}lder's inequality implies that 
\begin{equation}\label{ho}
\|fs(q)\|_1\le \|f|s(q)|^{1/p}\|_p\|s(q)\|_1^{1/q}\;.
\end{equation}
Since $\|q_x\|= 1$ for all $x$ it follows that $s(q)$ is bounded. Since it's also true that $s(q)\in L^1(M,d\mu),$ the result follows.
\end{proof}
\begin{definition}
We define an identity–preserving, continuous $^*$-linear map
\begin{equation}
    Q:\bigoplus_{1\le p\le \infty}L^p(M,d\mu)\to M_{\hbar}\;,\;\;Q_f=\int_M f(x)q_x\,d\mu(x)\;.
\end{equation}
\end{definition}
More generally, $Q$ is definable on functions $f$ such that $fs(q)\in L^1(M,d\mu)$ for all $s\in M_{\hbar*}.$
\begin{proposition}\label{non}
If $M\ne \{*\}$ then $M_{\hbar}$ is noncommutative. 
\end{proposition}
\begin{proof}
Let $x\in M.$ Since $y\mapsto q_y$ is continuous and $q_x^2=q_x\ne 0,$ it follows that there exists $y\ne x\in M$ such that $q_xq_y,\,q_yq_x\ne 0.$
Suppose $q_xq_y=q_yq_x.$
Then
\begin{equation}
q_xq_y=q_yq_xq_y=\lambda q_y\implies q_xq_y=\lambda q_xq_y\implies\lambda =1\implies q_xq_y=q_y.
\end{equation}
Similarly, it follows that $q_xq_y=q_x,$ therefore $q_x=q_y.$ This contradicts injectivity.
\end{proof}
\subsubsection{The Adjoint to the Quantization Map and the Noncommutative Product}\label{adjs}
The next proposition follows immediately from the definition.
\begin{proposition}
If $(M_{\hbar},q)$ has enough states, then $A\mapsto \rho_A$ is injective.
\end{proposition}
\begin{definition}
Let $\mathcal{H}_S$ denote the vector space of all $A\in M_{\hbar}$ such that
\begin{equation}
    \|A\|_{\mathcal{H}_S}:=\int_M \rho_{A^*A}\,d\mu<\infty\;.
\end{equation}
$\mathcal{H}_S$ is a Hilbert space with inner product
\begin{equation}
    \langle A,B\rangle_{\mathcal{H}_S}=\int_M\rho_{A^*B}\,d\mu\;.
\end{equation}
\end{definition}
If we make an identification $M_{\hbar}\cong \textup{B}(\mathcal{H}),$ then this is just the space of Hilbert–Schmidt operators.
\begin{proposition}
$Q\vert_{L^2}(L^2(M))\subset \mathcal{H}_S$ and $\rho\vert_{\mathcal{H}_S}(\mathcal{H}_S)\subset L^2(M)\cap L^{\infty}(M).$
\end{proposition}
\begin{lemma}
 $Q\vert_{L^2}$ and $\rho\vert_{\mathcal{H}_S}$ are adjoints. We write $\rho\vert_{\mathcal{H}_S}=Q^{\dagger}.$
 \begin{proof}
\begin{align}
    &\langle A, Q_f\rangle_{\mathcal{H}_S}=\int_M \rho_{A^*Q_f}(x)\,d\mu(x)=\int_{M\times M} f(y)\,\rho_{A^*q_y}(x)\,d\mu(y)\,d\mu(x)
    \\&=\int_{M\times M} f(y)\,\rho_{(Aq_x)^*}(x)\,d\mu(x)\,d\mu(y)=\int_M f(y)\,\rho_{A^*}(y)\,d\mu(y)=\langle \rho_{A},f\rangle_{L^2}\;,
\end{align}
 \end{proof}
 where the third equality follows from $\rho_x(q_yA)=\rho_y(Aq_x)$ and the fact that states are $^*$-linear. The fourth equality follows from $Q_{1}=\mathds{1}.$
\end{lemma}
\begin{lemma}
If $(M_{\hbar},q)$ has enough states then $Q\vert_{L^2}$ has dense image. Conversely, if $Q\vert_{L^2}$ has dense image then there are enough states to separate $\mathcal{H}_S.$ Furthermore, $Q\vert_{L^2}$ is injective if and only if $Q^{\dagger}$ has a dense image.
\end{lemma}
\begin{proof}
This follows from the fact that $Q^{\dagger}$ and $Q\vert_{L^2}$ are adjoint, since a map is injective if and only if its adjoint has dense image.
\end{proof}
\begin{definition}
We have a pairing 
\begin{equation}
\langle\cdot,\cdot\rangle:L^{\infty}(M)_{*}\times M_{\hbar*}\to\mathbb{C}\;,\;\;\langle d\lambda,s\rangle=\int_M s(q_x)\,d\lambda(x)\;,
\end{equation}
where $d\lambda$ is a Borel measure which is locally absolutely continuous with respect to the Lebesgue measure.
\end{definition}
\subsubsection{The 3-Point Function $\Delta$ and its Cohomology Class}\label{3point}
Here, we show that associated to any abstract coherent state quantization of $(M,d\mu)$ is a canonical representative of a class in $H^2(M,\mathbb{C}),$ which will actually turn out to be in $H^2(M,\mathbb{Z}).$ This is an abstraction of the first Chern class of a line bundle.
\\\\Recall that $\textup{Pair}\,M$ is the pair groupoid, \cref{pair}.
\begin{definition}
We define 
\begin{equation}
\Delta: \textup{Pair}^{(2)}\,M\to\mathbb{C}\;,\;\;\Delta(x,y,z)=\rho_x(q_yq_z)\;.\end{equation}
\end{definition}
We call $\Delta$ the 3–point function. This is an abstraction of the 3–point function defined in \cite{berc}. One can similarly define $n$–point functions for any $n\in\mathbb{N},$ but the 3–point function is the one that determines a first Chern class. As we are about to see, $\Delta$ determines a degree 2 class in the cohomology of the local pair groupoid, or equivalently, the Alexander–Spanier cohomology. First, we mention that $\Delta$ is invariant under isomorphisms:
\begin{proposition}
Let $(M_{\hbar,1},q_{1}),\,(M_{\hbar,2},q_{2})$ be abstract coherent state quantizations of $(M,d\mu).$ Suppose that $\pi:M_{\hbar,1}\to M_{\hbar,2}$ is an isomorphism of $W^*$-algebras for which $\pi(q_{1,x})=q_{2,x}.$ Then $\Delta_1(x,y,z)=\Delta_2(x,y,z).$
\end{proposition}
It follows from the definition that:
\begin{proposition}
$\Delta$ is $1$ on the identity bisection and is conjugation–antisymmetric with respect to the $S_3$–action on $\textup{Pair}^{(2)}\,M.$
\end{proposition}
\begin{definition}
Let $U$ be a neighborhood of $M\xhookrightarrow{}\textup{Pair}^{(2)}M$ such that $\Delta\vert_U$ is nowhere zero. We define a cohomology class in $H^2(M,\mathbb{C})$ as follows: let 
\begin{equation}
    [\Delta](x,y,z):=\frac{\Delta(x,y,z)}{\Delta(x,x,z)}\;.
\end{equation}
This defines a closed 2-cocycle in the cochain complex of local of the pair groupoid, valued in $\mathbb{C}^*.$ Its logarithm is a 2-cocycle valued in $\mathbb{C}$ (where we choose the logarithm so that $\log{[\Delta]}\vert_M=0).$ 
\end{definition}
One can see that this is a cocycle by applying the groupoid differential and using that $\rho_x(q_yq_z)q_x=q_xq_yq_zq_x,$ together with conjugation–antisymmetry. Assuming $(x,y,z)\mapsto \rho_x(q_yq_z)$ is smooth, we get a closed 2-form by applying the van Est map, \cref{vanest}. 
\begin{remark}
$\Delta$ implicitly appeared earlier, since
\begin{equation}
    \langle Q_f,Q_g\rangle_{\mathcal{H}_S}=\int_{M^3}\overline{f(y)}g(x)\Delta(x,y,z)\,d\mu(x)\,d\mu(y)\,d\mu(z)\;.
    \end{equation}
\end{remark}
\subsection{Abstract Quantization Maps}
Sometimes abstract quantization maps aren't assumed to be determined by an abstract coherent state quantization, eg. \cite{rieffel}. However, if a quantization map $Q:L^{\infty}(M)\to M_{\hbar}$ is determined by an abstract coherent state quantization, then the latter is uniquely determined. This happens for well–behaved quantization maps which ``preserve minimal projections". This implies that being an abstract coherent state quantization is a property of a quantization map.
\\\\We'll assume $M$ has finite volume, for simplicity.
\begin{proposition}
Let $M$ be a manifold with Borel measures $d\mu,\,d\mu',$ with finite volume. Let $q,\,q':M\to M_{\hbar}$ be abstract coherent state quantizations with respect to $d\mu,\,d\mu',$ respectively. Suppose that $Q=Q'.$ Then $d\mu=d\mu$ and $q=q'.$
\end{proposition}
\begin{proof}
Without loss of generality, assume $M_{\hbar}=B(\mathcal{H}).$ For all $x\in M,$ $\textup{Tr}(q_{x})=\textup{Tr}(q_{x}')=1.$ It follows that for all $f\in C_c(M),$
\begin{equation}
    \int_M f(x)\,d\mu(x)=\int_M f\,d\mu'(x)\;.
\end{equation}By the Riesz–Markov–Kakutani representation theorem, $d\mu=d\mu'.$ This implies that for all $f\in C_c(M)$
\begin{equation}
    \int_M f(x)q_{x}\,d\mu(x)= \int_M f(x)q_{x}'\,d\mu(x)\;,
\end{equation}
from which the result follows.
\end{proof}
\begin{corollary}
Let $M$ be a manifold and consider a map $Q:L^{\infty}(M)\to M_{\hbar}.$ If $Q$ is determined by an abstract coherent state quantization, then $q$ and $d\mu$ are uniquely determined.
\end{corollary}
To understand how $q$ and $d\mu$ are constructed, assume $Q$ is determined by an abstract coherent state quantization. Then without loss of generality, we may assume $M_{\hbar}\cong \textup{B}(\mathcal{H}).$ Consider the map
\begin{equation}
C_c(M)\to\mathbb{C}\;,\;\;f\mapsto \textup{Tr}(Q_f)\;.
\end{equation}
By the Riesz–Markov–Kakutani representation theorem there is a Borel measure $d\mu$ on $M$ such that
\begin{equation}
    \textup{Tr}(Q_f)=\int_M f\,d\mu\;.
\end{equation}
$Q$ must extend to a bounded map
\begin{equation}
    Q\vert_{L^2}:L^2(M,d\mu)\to \textup{B}(\mathcal{H})_{\textup{HS}}
\end{equation}
with respect to the Hilbert–Schmidt norm. Letting $Q^{\dagger}$ be its adjoint, for each $x\in M$ we get a bounded map
\begin{equation}
\textup{B}(\mathcal{H})_{\textup{HS}}\to\mathbb{C}\;,\;\;A\mapsto (Q^{\dagger}(A))(x)\;.
\end{equation}
The Riesz representation theorem determines a Hilbert-Schmidt operator $q_x$ such that
\begin{equation}
    (Q^{\dagger}(A))(x)=\textup{Tr}(A^*q_x)\;.
\end{equation}
By assumption, $q_x$ must be a minimal projection.
\subsection{Abstract Deformation Quantization}
\begin{definition}
Let $I\subset (0,1]$ contain $0$ an accumulation point. An abstract coherent state \textbf{deformation} quantization of a symplectic manifold $(M^{2n},\omega)$ is given by an abstract coherent state quantization of $(M,d\mu_{\hbar}=\omega^n_{\hbar})$ for each $\hbar \in I,$ such that 
\begin{enumerate}
        \item [4.]$\omega_{\hbar}$ is a symplectic form such that $\hbar \omega_{\hbar}\xrightarrow[]{\hbar\to 0} C\omega$ pointwise, for some $C>0.$
        \item [5.]$\rho_{Q_{\hbar}(f)}(x)\xrightarrow[]{\hbar\to 0}f(x).$
     \item[6.] There exists a formal deformation quantization $\star_{\hbar}$ on $C_c^{\infty}(M)[[\hbar]]$ such that, for all $n\in\mathbb{N},$
\begin{align}\label{formal}
\frac{1}{\hbar^{n}}||Q_{\hbar}(f)Q_{\hbar}(g)-Q_{\hbar}(f\star_{\hbar}^n g)||_{\hbar}\xrightarrow[]{\hbar \to 0} 0\;,
\end{align}
where $f\star_{\hbar}^n g$ is the truncation of the formal deformation quantization above order $n.$
\end{enumerate}
\end{definition}
In nice cases, $\hbar_2>\hbar_1\implies\ker{Q_{\hbar_2}}\subset\ker{Q_{\hbar_1}}$ and the perturbative expansion of the resulting non–commutative product on $(\ker{Q}_{\hbar})^{\perp}$ is a star product.  See \cite{mar}, page 25.
\begin{remark}
The formal deformation quantization condition \ref{formal} is really just the corresponding condition for $n=1$ together with a smoothness condition (proposition 2.2, \cite{eli2}). That is, if there exists linear maps $C_k:C_c^{\infty}\otimes C_c^{\infty}(M)\to C_c^{\infty}(M),\,k=1,2,\ldots$ such that
\begin{equation}
  Q_{\hbar}(f)Q_{\hbar}(g)\sim Q_{\hbar}\big(fg+\sum_{k=0}^\infty(i\hbar)^k C_k(f,g)\big)\;\;\textup{as }\hbar\to 0\;,
\end{equation}
then $C_k,\,k=1,2,\ldots$ are unique and 
\begin{equation}
    f\star_{\hbar} g=fg+\sum_{k=0}^\infty(i\hbar)^k C_k(f,g)
    \end{equation}
is a formal deformation quantization if
\begin{equation}
    C_1(f,g)-C_1(g,f)=\{f,g\}\;.
\end{equation}
\end{remark}
\section{Formal Path Integral Quantization of Symplectic Manifolds}
Before doing rigorous mathematics we will describe the non–perturbative aspects of the formal path integral theory, much of which is described in \cite{klauder}, \cite{klauder2}. Some of the perturbative aspects are discussed in \cite{grady}, also \cite{brane}.
\\\\
In \cite{feynman}, Feynman formulates quantum mechanics as a path integral over a space of maps into phase space. This can be formally generalized to symplectic manifolds: given a prequantum line bundle 
\begin{equation}
    (\mathcal{L},\nabla,\langle \cdot,\cdot\rangle)\to (M^{2n},\omega/\hbar)\;,
\end{equation}
the path integral is given by
\begin{equation}
    \int \mathcal{D}\gamma\, P(\gamma)e^{-\frac{i}{\hbar}\int_0^1 \gamma^*H\,dt}\;,
\end{equation}
where $P(\gamma)$ is the parallel transport map between the fibers of $\mathcal{L}_{\hbar}$ over $\gamma(0),\gamma(1),$ and where $H$ is a Hamiltonian $H:M\to\mathbb{R}.$ If we let $H=0$ and take the domain of integration to be the space of maps
\begin{equation}
    \{\gamma:[0,1]\to M: \gamma_0=x,\gamma_1=y\}
\end{equation}
we get the coherent state path integral, which is equivalent to the equal–time path integral for $H\ne 0:$
\begin{equation}\label{path}
\int_{\gamma_0=x}^{\gamma_1=y} \mathcal{D}\gamma\, P(\gamma)\;\in \textup{Hom}(\mathcal{L}_x,\mathcal{L}_y)\;.
\end{equation}
Formally, this path integral is a Hermitian section of
\begin{equation}
    \pi_0^*\mathcal{L}^*\otimes \pi_1^*\mathcal{L}\to M\times M
\end{equation}
and is equal to the identity map when $x=y.$ In addition, it has the property that integrating over all paths that go from $x$ to $y$ gives the same result as first integrating over all paths that go from $x$ to $z$ and then from $z$ to $y,$ and finally integrating over $z.$ These properties imply that the linear operator which takes a section $\Psi\in\Gamma(\mathcal{L})$ to the section
\begin{equation}
  x\mapsto \int_M \frac{\omega^n_z}{\hbar^n}\,\Psi(z)\int_{\gamma_0=z}^{\gamma_1=x}\mathcal{D}\gamma\,P(\gamma)
\end{equation}
is an orthogonal projection. Its image is the physical Hilbert space $\mathcal{H}_{\textup{phy}}.$ 
\\\\For $0<t<1$ and $f\in L^{\infty}(M),$ there is a map on the space of paths given by $f_t(\gamma)=f(\gamma(t)).$
The integral kernel of the quantization map is given by
\begin{align}
   & Q:L^{\infty}(M)\to \Gamma(\pi_0^*\mathcal{L}^*\otimes \pi_1^*\mathcal{L})\;,\;\;Q_{f}(x,y)=\int_{\gamma_0=x}^{\gamma_1=y}\mathcal{D}\gamma\,P(\gamma)f_t(\gamma)\;,
\end{align}
and it satisfies $Q_{f}Q_{g}\sim Q_{f\star_{\hbar} g}$ as $\hbar\to 0,$ where $\star_{\hbar}$ is a star product (\cite{brane}, \cite{grady}). Each $x\in M$ defines a one–dimensional orthogonal projection, called a coherent state, resulting in a map
\begin{equation}
   q:M\to P(\mathcal{H}_{\textup{phy}})\subset B(\mathcal{H}_{\textup{phy}})\;,\;\;x\mapsto q_x\;,
    \end{equation}
where \begin{equation}
   q_x(\Psi)(y)= \Psi(x)\int_{\gamma_0=x}^{\gamma_1=y}\mathcal{D}\gamma\,P(\gamma)\;.
\end{equation}
This is the operator associated to the integral kernel $Q_{\delta_x},$ where $\delta_x$ is the delta function supported at $x.$ The map $q$ is a classifying map for $\mathcal{L}\to M,$ whose first Chern class is determined by applying the van Est map (\cref{vanest}) to the 3–point function
\begin{equation}
    \textup{Pair}^{(2)}\,M=M^3\to\mathbb{C}\;,\;\;(x,y,z)\mapsto \int_{\gamma_0=x,\gamma_1=y,\gamma_\infty=z}\mathcal{D}\gamma\,P(\gamma)\;,
\end{equation}
where the integral is over $\{\gamma:S^1\to M:\gamma_0=x,\,\gamma_1=y,\,\gamma_\infty=z\},$ where $0,1,\infty \in S^1$ are any distinct points ordered counter–clockwise.
\begin{remark}
Assuming $q$ is an immersion, computing \cref{path} endows $M$ with a Riemannian metric. If furthermore $q$ embeds $M$ as an almost complex submanifold, then $M$ inherits the structure of a K\"{a}hler manifold.
\end{remark}
\begin{remark}
Formally, integrating the coherent states over the fibers of a Lagrangian polarization determines an orthogonal basis for the physical Hilbert space, and this is equivalent to integrating the endpoints of \cref{path} over the fibers of the Lagrangian polarization. This gives a common form of the phase space path integral, \cite{ber}. One can see that that this works in the case of $\textup{T}^*\mathbb{R}^n$ when using linear Lagrangian polarizations, ie. using $x=\textup{const.}$ gives the position space basis of the Hilbert space.
\end{remark}
\subsection{Formal Path Integral Quantization of Poisson Manifolds}
The path integral quantization of symplectic manifolds generalizes to Poisson manifolds $(M,\hbar\Pi)$ by replacing $M\times M$ with the symplectic groupoid $\Pi_1(T^*M),$ and the prequantum line bundle over $M$ with the multiplicative prequantum line bundle $(\mathcal{L},\nabla,\langle\cdot,\cdot\rangle)\to\Pi_1(T^*M).$ The path integral is formally given by
\begin{equation}\label{pathalg}
   \int_{[\gamma]= g}\mathcal{D}\gamma\,P(\gamma)\;\in \mathcal{L}_g\;.
    \end{equation}
Here, $P(\gamma)$ denotes parallel transport over the algebroid path $\gamma$ with respect to the multiplicative prequantum line bundle, and the path integral is over algebroid paths with homotopy class $g\in \Pi_1(T^*M).$ This reduces to the path integral of symplectic manifolds if $\Pi$ is symplectic and the pair groupoid is used.
\\\\
The quantization $Q_{f}$ of $f\in L^{\infty}(M)$ is a section of the prequantum line bundle — these sections generate a $W^*$-algebra. Each $x\in M$ determines a state $\rho_x$ such that $\rho_x(Q_{f})=Q_f(x),$ where we are identifying $x$ with its identity arrow. There are corresponding elements $q_x$ that resolve the identity, but they need not be projections. Let $0,\,\infty\in \partial D$ be distinct points on the boundary of a disk. If the source simply connected groupoid is used, then formally 
\begin{equation}\label{state}
Q_f(x)=\int_{X:TD\to T^*M}\mathcal{D}X\,f(X(\infty))e^{\frac{i}{\hbar}\int_D X^*\Pi}\;,
\end{equation}
where the path integral is over algebroid morphisms for which $X(0)=x.$
\\\\Integrals of the form \cref{pathalg}, \cref{state} are directly related to Kontsevich's star product \cite{kontsevich} via the Poisson sigma model \cite{bon}, which formally explains why they perturbatively produce star products, ie. $Q_{f_1}Q_{f_2}\sim Q_{f_1\star_{\hbar} f_2}$ as $\hbar\to 0.$
\section{Path Integral Quantization of Manifolds}\label{pman}
We will first define the general case of a path integral whose integrand is parallel transport between two fibers of a line bundle, ie.
\begin{equation}\label{path1}
    \Omega(x,y)=\int_{\gamma(0)=x}^{\gamma(1)=y}P(\gamma)\,\mathcal{D}\gamma\;,
\end{equation}
where $P(\gamma)$ denotes parallel transport over $\gamma.$  We do this before discussing the path integral in the symplectic case, because that case involves defining a sequence of path integrals depending on $\hbar$ and obeying an asymptotic condition.
\\\\Given a line bundle $\pi:\mathcal{L}\to M,$ any section $\Omega$ of 
\begin{equation}
    \pi_0^*\mathcal{L}^*\otimes \pi_1^*\mathcal{L}\to M\times M
    \end{equation}
which satisfies $\Omega(x,x)=\mathds{1}$ for all $x\in M$ determines a connection $\nabla_{\Omega}$ on $\mathcal{L}.$ This is because we can identify a connection with a splitting 
\begin{equation}\label{split}
 \begin{tikzcd}
T\mathcal{L} \arrow[r] & \pi^*TM \arrow[l, "\nabla", dotted, bend left, shift left=2]
\end{tikzcd},
\end{equation}
and $\Omega$ determines such a splitting.
\begin{definition}
Given $\Omega\in\Gamma(\pi_0^*\mathcal{L}^*\otimes \pi_1^*\mathcal{L})$ such that $\Omega(x,x))=\mathds{1}$ for all $x\in M,$ we define $\nabla_{\Omega}$ to be the connection determined by the splitting of \cref{split} given by 
\begin{equation}
    \nabla_{\Omega}(l_x,X)=(l_x\Omega(x,\cdot))_*X\;,
\end{equation}
where $(l_x,X)\in\pi^*T_x M$ and $l_x\Omega(x,\cdot)\in\Gamma(\mathcal{L})$ is given by
\begin{equation}
    y\mapsto l_x\Omega(x,y)\;.
\end{equation}
\end{definition}
\subsection{Definition of the Path Integral}
In \cref{theproof}, we formally prove that finding $\Omega$ satisfying the following definition is equivalent to computing the path integral. In particular, we will rigorously prove that such an $\Omega$ determines an inverse system of complex measures on the space of paths.
\begin{definition}\label{patho}
Let $\mathcal{L}\to M$ be a line bundle with Hermitian metric and let $d\mu$ be a Borel measure on $M.$ Let $\Omega$ be a continuous section of
\begin{equation}
    \pi_0^*\mathcal{L}^*\otimes \pi_1^*\mathcal{L}\to M\times M
    \end{equation}
such that 
\begin{enumerate}
    \item \label{nor}$\Omega(x,x)=\mathds{1}\,,$
    \item \label{norm2} $|\Omega(x,y)|<1$ if $x\ne y\,,$
    \item\label{orth} $\Omega(x,y)=\Omega^*(y,x)\,,$
    \item \label{conditionf}$\int_M \Omega(x,z)\Omega(z,y)\,d\mu(z)=\Omega(x,y)\,,$
    \item \label{bounded} $\sup_{x\in M} \int_M|\Omega(x,y)|\,d\mu(y)<\infty\,.$ 
\end{enumerate}
We say that $\Omega$ is an (equal–time) propagator. If $\Omega$ is smooth and $\nabla_{\Omega}=\nabla,$ then we say that $\Omega$ is a propagator integrating $\nabla.$ If $F(\nabla_{\Omega})=\omega,$ we say that $\Omega$ integrates $\omega.$
\end{definition}
In the following, $\textup{Pair}\,M$ is the pair groupoid, \cref{pair}.
\begin{definition}
We define a function
\begin{equation}
\Delta:\textup{Pair}^{(2)}\,M\to\mathbb{C}\;,\;\; \Delta(x,y,z)=\Omega(x,y)\Omega(y,z)\Omega(z,x)
\end{equation}
which we call the 3–point function. 
\end{definition}
Associated to $\Delta$ is a cohomology class in $H^2(M,\mathbb{C}),$ where we use the identification of $H^{\bullet}(M,\mathbb{C})$ with the cohomology of the local pair groupoid:
\begin{definition}
Let $U$ be a neighborhood of $M\xhookrightarrow{}\textup{Pair}^{(2)}\,M$ such that $\Delta\vert_U$ is nowhere zero. We have a function
\begin{equation}
[\Delta]:U\to\mathbb{C}^*\;,\;\;[\Delta](x,y,z)=\frac{\Delta(x,y,z)}{\Delta(x,x,z)}\;.
\end{equation}
$[\Delta]$ is a $2$–cocycle on the local pair groupoid valued in $\mathbb{C}^*,$ and $\log{[\Delta}]$ is a 2-cocycle valued in $\mathbb{C},$ where we choose the logarithm so that $\log{[\Delta}]\vert_M=0.$
\end{definition}
In the following, $\textup{VE}$ is the van Est map (\cref{vanest}):
\begin{proposition}
$F(\nabla_{\Omega})=\textup{VE}(\log{\Delta}).$
\end{proposition}
When quantizing symplectic manifolds it is mostly $\omega$ that we are interested in, rather than the specific connection.
\\\\We call $\Delta$ a 3–point function because as a path integral it is equal to
\begin{equation}
    \Delta(x,y,z)=\int_X P(\gamma)\,\mathcal{D}\gamma\;,\;\;X=\{\gamma:S^1\to M: \gamma(0)=x, \gamma(1)=y, \gamma(\infty)=z\}\;,
\end{equation}
where $0,\,1,\,\infty\in S^1$ are any distinct points ordered counterclockwise.
\begin{remark} 
As we will see, associated to any $x\in M$ is a state, and condition $1$ is required for this state to be normalized. Condition $2$ is a mild condition, and it means that this state is localized at $x$ and that the map from points to states is injective. The third condition is a Hermitian condition and is due to parallel transport defining a Hermitian map. The fourth condition is a consistency condition and implies that the path integral satisfies a form of Fubini's theorem: integrating from $x$ to $y$ is equivalent to integrating from $x$ to $z$ and then from $z$ to $y,$ and finally integrating over $z.$ Condition 5 is a technical condition and can be weakened, but it automatically holds if $\textup{Vol}_{\mu}(M)<\infty.$
\end{remark}
Note that, conditions 2 and 4 imply that for all $x,y\in M,$ 
\begin{equation}
    (z\mapsto\Omega(x,z)\Omega(z,y))\in L^1(M,\mathcal{L})\;.
    \end{equation}
Therefore, condition 5 makes sense. Furthermore, condition 2 is a mild condition: the assumption that 
\begin{equation}
    (z\mapsto\Omega(x,z)\Omega(z,y))\in L^1(M,\mathcal{L})
    \end{equation}
together with conditions 1, 3, 4 imply that $|\Omega(x,y)|\le 1,$ via the Cauchy–Schwarz inequality.
\begin{lemma}\label{boundedd}
Let $\Psi\in\Gamma(\mathcal{L})$ be a continuous section. Then
\begin{equation}
M\to L^2(M,\mathcal{L})\;,\;\;x\mapsto \Psi(x)\Omega(x,\cdot)\
\end{equation}
is a continuous.
\end{lemma}
\begin{proof}
Properties 1, 3, 5 show that
\begin{align}
    \int_M |\Psi(x)\Omega(x,z)-\Psi(y)\Omega(y,z)|^2\,d\mu(z)=|\Psi(x)|^2+|\Psi(y)|^2-\langle\Psi(x),\Psi(y)\rangle-\langle\Psi(y),\Psi(x)\rangle\;,
\end{align}
and the result follows.
\end{proof}
\begin{lemma}\label{norma}
For all $x\in M,$ $|\Omega(x,y)|^2$ is a probability density, ie. $\int_M |\Omega(x,y)|^2\,d\mu(y)=1.$
\end{lemma}
\begin{proof}
This follows from conditions 1, 3, 5.
\end{proof}
\begin{lemma}
The map
\begin{equation}\label{rightside}
    P_{\Omega}:L^2(M,\mathcal{L})\to L^2(M,\mathcal{L})\;,\;\;P_{\Omega}(\Psi)(x)=\int_M\Psi(y)\Omega(y,x)\,d\mu(y)
\end{equation}
is well–defined and is an orthogonal projection. In particular, it is bounded.
\end{lemma}
\begin{proof}
First we'll show that for all $x\in M,$ $(y\mapsto(\Psi(y)\Omega(y,x))\in L^1(M,\mathcal{L}_x),$ so that the right side of \cref{rightside} makes sense:
\begin{align}
\int_M|\Psi(y)\Omega(y,x)|\,d\mu(y)\le \int_M|\Psi(y)|^2\,d\mu(y)\int_M|\Omega(y,x)|^2\,d\mu(y)=\int_M|\Psi(y)|^2\,dy\;,
\end{align}
where the inequality follows from H\"{o}lder's inequality and the equality follows from \cref{norma}
\\\\Now we'll show that $P_{\Omega}\Psi\in L^2(M,\mathcal{L}).$ We have that 
\begin{align}
&\nonumber\|P_{\Omega}\|_2^2=\int_M \Big|\int_M \Psi(y)\Omega(y,x)\,d\mu(y)\Big|^2\,d\mu(x)
\\&\le \int_{M}\bigg(\int_M|\Psi(y)|^2|\Omega(y,x)|\,d\mu(y)\int_M |\Omega(y,x)|\,d\mu(y)\bigg)\;d\mu(x)\le \|\Psi\|^2_2\sup_{x\in M}\|\Omega(x,\cdot)\|^2_1<\infty\;,
\end{align}
where the first inequality follows the triangle inequality and H\"{o}lder's inequality with the functions $|\Psi|\sqrt{|\Omega(x,\cdot)|},\,\sqrt{|\Omega(x,\cdot)|},$ and the last inequality follows from Fubini's theorem, condition 3 and condition 5.
\\\\Finally, we'll show that $P_{\Omega}^2=P_{\Omega},\,P_{\Omega}^*=P_{\Omega}.$ We first note that if $\Psi\in L^2(M,\mathcal{L})$ then $(y\mapsto P_{\Omega}\Psi(y)\Omega(y,x))\in L^1(M,\mathcal{L})$ for each $x\in M.$ To see this, note that
\begin{equation}
\int_M |P_{\Omega}\Psi(y)\Omega(y,x)|\,d\mu(y)\le\|P_{\Omega}\Psi\|_2\|\Omega(\cdot,x)\|_2=\|P_{\Omega}\Psi\|_2\;.
\end{equation}
Therefore, Fubini's theorem and condition 3 imply that $P^2=P.$ Now observe that for $\Psi_1,\,\Psi_2\in L^2(M,\mathcal{L}),$ $((x,y)\mapsto \langle\Psi_1(x)\Omega(x,y),\Psi_2(y)\rangle)\in L^1(M\times M,\mathbb{C}).$ This is because
\begin{align}
&\nonumber \int_{M\times M}|\langle\Psi_1(x)\Omega(x,y),\Psi_2(y)\rangle|\,d\mu(x) d\mu(y)
\\&\le \int_{M\times M}|\Psi_1(x)||\Omega(x,y)||\Psi_2(y)|\,d\mu(x) d\mu(y)
\le \|\Psi_1\|^2_2\|\Psi_2\|^2_2\sup_{x\in M}\|\Omega(x,\cdot)\|^2_1\;.
\end{align}
Therefore, by Fubini's theorem we can switch the order of integration and use condition 3 to see that $P_{\Omega}^*=P_{\Omega}.$ Since orthogonal projections have norm 1, this completes the proof.
\end{proof}
\section{Propagators $\xleftrightarrow[]{}$ Abstract Coherent State Quantization}
Here we will show that there is a an equivalence of categories between the category of (equal–time) propagators and the category of abstract coherent state quantizations.
More precisely:
\begin{theorem}\label{equivalence}
Let $(M,d\mu)$ be a manifold with a Borel measure.
There is a $\Delta$–preserving equivalence of categories between:
\begin{enumerate}
    \item Abstract coherent state quantizations $(M_{\hbar},q)$ such that
    \begin{equation}\label{extrac}
        \sup_{x\in M}\int_M \sqrt{\rho_x(q_y)}\,d\mu(y)<\infty\;\footnote{This is automatically satisfied for manifolds with finite volume.}
    \end{equation}
    \item Propagators $\Omega.$
\end{enumerate} 
\end{theorem}
In the $1\to 2$ direction of this equivalence, the cohomology class determined by the 3-point function of $(M_{\hbar},q)$ is the first Chern class of the line bundle associated to it, which is the pullback of the canonical bundle. In the $2\to 1$ direction, $M$ is embedded into $\textup{P}(\mathcal{H}_{\textup{phy}})$ and $(\mathcal{L},\nabla_{\Omega})\to M$ is the pullback of the canonical bundle with the Fubini–Study connection.
\\\\We will take the morphisms in the underlying categories to simply be isomorphisms:
\begin{definition}
    Given two abstract coherent state quantizations of $(M_{1,\hbar},q_1),\,(M_{2,\hbar},q_2)$ of $(M,d\mu),$ an isomorphism between them is a morphism of $W^*$–algebras 
\begin{equation}
    \pi:M_{1,\hbar}\to M_{2,\hbar}
\end{equation}
such that $q_2=\pi\circ q_1.$ 
\end{definition}
Similarly,
\begin{definition}
An isomorphism of (equal–time) propagators is a fiberwise isometry of line bundles that fiberwise commutes with the propagators.
\end{definition}
\subsection{Propagator $\to$ Abstract Coherent State Quantization}
We discuss some basic properties of the propagator and its associated representation. After doing this we will deduce the forward direction of the desired equivalence.
\begin{definition}
We define the physical Hilbert space $\mathcal{H}_{\textup{phy}}$ to be the image of $P_{\Omega}\,.$
\end{definition}
\begin{proposition}\label{ae}
For all $\Psi\in\mathcal{H}_{\textup{phys}},$ $\Psi$ is essentially bounded and has a continuous representative.
\end{proposition}
\begin{proof}
This follows from 
\begin{equation}
    \Psi(\cdot)=\int_M \Psi(y)\Omega(y,\cdot)\,d\mu(y)\;.
\end{equation}
That the right side is continuous follows from \cref{boundedd}. That $\Psi$ is essentially bounded follows from the next proposition, since 
\begin{equation}
    |\Psi(x)|=|\Psi(x)|\|\Omega(x,\cdot)\|_2=\|q_x\Psi\|_2\le \|\Psi\|_2\;,
\end{equation}
where the final inequality follows from the fact that orthogonal projections have norm $1.$
\end{proof}
In particular, this means that pointwise evaluation of sections in $\mathcal{H}_{\textup{phy}}$ is continuous. Such Hilbert spaces are commonly called reproducing kernel Hilbert spaces. 
\begin{proposition}
For each $z\in M,$ the operator $q_z\in B(\mathcal{H}_{\textup{phy}})$ given by
\begin{equation}\label{make}
   q_z\Psi(x)=\Psi(z)\Omega(z,x)
\end{equation}
is a rank-one orthogonal projection, with eigenspace given by $x\mapsto l_z\Omega(z,x)$ for $l_z\in \mathcal{L}_z.$
\end{proposition}
Here, we are assuming $\Psi$ is a continuous representative to make sense of \cref{make}.
\begin{proof}
Let $l_z\in\mathcal{L}_z$ be normalized. We can see that $q_z$ is an orthogonal projection onto $x\mapsto l_z\Omega(z,x),$ since 
\begin{align}
    & l_z\Omega(z,x)\int_M \langle l_z\Omega(z,y),\Psi(y)\rangle_{\mathcal{L}_y} \,d\mu(y)
    \\& =l_z\Omega(z,x)\int_M \langle l_z,\Psi(y)\Omega(y,z)\rangle_{\mathcal{L}_z} \,d\mu(y)\;\;\textup{ (by condition }\ref{orth})
    \\&=l_z\Omega(z,x)\,\langle l_z,\int_M \Psi(y)\Omega(y,z)\, d\mu(y)\rangle_{\mathcal{L}_z}
    \\& =l_z\Omega(z,x)\,\langle l_z,\Psi(z)\rangle_{\mathcal{L}_z}\;\;\textup{ (since } \Psi\in\mathcal{H}_{\textup{phy}})
    \\&=\Psi(z)\Omega(z,x)\;.
\end{align}
\end{proof}
\begin{lemma}\label{wea}
For $\Psi\in\mathcal{H}_{\textup{phy}},$    $\langle\Psi,q_x\Psi\rangle=|\Psi(x)|^2.$
\end{lemma}
\begin{proof}
\begin{align}
&\langle\Psi,q_x\Psi\rangle= \int_M \langle\Psi(y),\Psi(x)\Omega(x,y)\rangle\,d\mu(y)
  = \int_M\langle\Psi(y)\Omega(y,x),\Psi(x)\rangle\,d\mu(y)
  =|\Psi(x)|^2.
\end{align}
\end{proof}
The following shows that, indeed, the eigenvectors of $q_x$ are maximally localized at $x.$
\begin{corollary}
Let $\Psi\in\mathcal{H}_{\textup{phy}}$ be normalized. Then $|\Psi(x)|^2\le 1$ and $|\Psi(x)|^2= 1$ if and only if $\Psi$ is an eigenvector of $q_x.$
\end{corollary}
\begin{proof}
This follows from \cref{wea} and the Cauchy–Schwarz inequality, using that $q_x$ has norm 1 and noting that equality in Cauchy–Schwarz occurs if and only if $q_x\Psi=\lambda\Psi.$
\end{proof}
\begin{proposition}
The map $M\to B(\mathcal{H}_{\textup{phys}}),\,x\mapsto q_x$ is continuous.
\end{proposition}
\begin{proof}
This follows from \cref{boundedd} and the fact that the map $\mathcal{H}\to \textup{B}(\mathcal{H})$ taking a vector to its associated orthogonal projection is continuous.
\end{proof}
\begin{proposition}
The map \begin{equation}
  Q\vert_{L^{\infty}(M)}:L^{\infty}(M)\to B(\mathcal{H}_{\textup{phys}})\;,\;\;Q_f=\int_M f(z)q(z)\,d\mu(z)  
\end{equation} 
is equal to $P_{\Omega}M_f\,,$ where $M_f\in B(L^2(M,\mathcal{L}))$ is the multiplication operator $\Psi\mapsto f\Psi.$
\end{proposition}
\begin{proof}
We can see this by noting that $M_f$ is a bounded on $L^p(M,\mathcal{L})$ and writing out the operator:
\begin{equation}
    Q_f\Psi=\int_M f(z)\Psi(z)\Omega(z,x)\,d\mu(z)=P_{\Omega}M_f\Psi\;.
\end{equation}
\end{proof}
Therefore, in the context of  Berezin–Toeplitz quantization, $Q_f$ is the  Berezin–Toeplitz operator. 
\begin{lemma}
$\mathcal{H}_{\textup{phy}}$ is an irreducible representation of the $W^*$–algebra that is weakly generated by the image of $q.$ 
\end{lemma}
\begin{proof}
Let $V\subset \mathcal{H}_{\textup{phy}}$ be a subrepresentation. By \cref{non}, if there exists $x\in M$ such that $q_xV\ne \{0\},$ then $q_xV\ne \{0\}$ for all $x\in M.$ Suppose that $q_xV=\{0\}$ for all $x\in M.$ Then for $\Psi\in V,$ $\Psi(x)=0$ for all $x,$ implying that $V=\{0\}.$ Otherwise, $V=\mathcal{H}_{\textup{phy}}$ since $V$ must contain the eigenvectors of $q_x$ for all $x\in M,$ and these generate $\mathcal{H}_{\textup{phy}}.$ 
\end{proof} 
\begin{corollary}\label{weak}
The $W^*$–algebra weakly generated by the image of $q$ is $B(\mathcal{H}_{\textup{phy}}).$
\end{corollary}
\begin{proof}
This is due to the fact that this $W^*$–algebra is a subspace of $B(\mathcal{H}_{\textup{phy}})$ that contains a compact operator (since it contains a rank–one projection), and since $\mathcal{H}_{\textup{phy}}$ is an irreducible representation of this $W^*$–algebra it must therefore contain all compact operators. Since a $W^*$-algebra containing all compact operators contains all bounded operators, the result follows. 
\end{proof}
\begin{remark}
For a simply connected symplectic manifold $(M,\omega),$ we formally have that 
\begin{equation}
\rho_{Q_f\ast Q_g}(x)=\int_{X\to M}\mathcal{D}X\,f(X(0))g(X(1))e^{\frac{i}{\hbar}\int_D X^*\omega}\;,
\end{equation}
where $0,\,1,\,\infty$ are cyclically ordered points on $\partial D$ and the path integral is over maps $X:D\to M$ such that $X(\infty)=x.$
\end{remark}
\subsubsection{Proof of $\to$ Direction of Equivalence}
Our $W^*$-algebra is $B(\mathcal{H}_{\textup{phy}})$ and it is weakly generated by $q,$ as shown in \cref{weak}. Furthermore, \cref{extrac} is satisfied because
\begin{equation}
    \sqrt{\rho_x(q_y)}=|\Omega(x,y)|\;.
\end{equation}
Also, $\Omega(x,y)\Omega(y,z)\Omega(z,x)=\rho_x(q_yq_z),$ so that the 3–point functions agree. We still need to show that:
\begin{lemma}
Weakly, 
\end{lemma}
\begin{equation}
\mathds{1}=\int_M q_x\,d\mu(x)\;.
\end{equation}
\begin{proof}
Let $\Psi_i\in \mathcal{H}_{\textup{phy}},\,i=1,2,\ldots$ be such that $\sum_{i=1}^{\infty}\|\Psi_i\|^2<\infty.$ Then using Fubini's theorem and \cref{wea},
\begin{align}
  &\int_M \sum_{i=1}^{\infty} |\langle\Psi_i,q_x\Psi_i\rangle|\,d\mu(x)=\sum_{i=1}^{\infty}\int_M  |\langle\Psi_i,q_x\Psi_i\rangle|\,d\mu(x)=
  \sum_{i=1}^{\infty} \|\Psi_i\|^2<\infty\;.
\end{align}
Note that, $\langle\Psi_i,q_x\Psi_i\rangle\ge 0.$ Therefore, this also shows that
\begin{align}
 \int_M \sum_{i=1}^{\infty} \langle\Psi_i,q_x\Psi_i\rangle\,d\mu(x)=\sum_{i=1}^{\infty} \|\Psi_i\|^2\;,
\end{align}
and this completes the proof.
\end{proof}
Finally, to complete this section we observe the following:
\begin{proposition}
If $h:\mathcal{L}_1\to\mathcal{L}_2$ is a fiberwise isometry of line bundles over $M$ and $\Omega_1,\,\Omega_2$ are propagators such that \begin{equation}
    \Omega_1(x,y)h(y)=h(x)\Omega_2(x,y)\;,
\end{equation}
then
\begin{equation}
B(\mathcal{H}_{\textup{phy},1})\to B(\mathcal{H}_{\textup{phy},2})\;,\;\;A\mapsto \big(\Psi\mapsto hAh^{-1}\Psi\big)
\end{equation}
is an isomorphism of $W^*$-algebras such that $hq_{1}h^{-1}=q_{2}.$
\end{proposition}

\subsection{Propagator $\xleftarrow[]{}$ Abstract Coherent State Quantization}
By \cref{important}, we may assume that $M_{\hbar}=\textup{B}(\mathcal{H}).$
\begin{definition}
Let $(\mathcal{L},\langle\cdot,\cdot\rangle)\to M$ be the line bundle with Hermitian metric obtained by pulling back $\mathcal{H}\setminus \{0\}\to \textup{P}(\mathcal{H})$ via $\pi\circ q.$
\end{definition}
We now define the propagator:
\begin{definition}
Define $\Omega:\pi_0^*\mathcal{L}^*\otimes\pi_1^*\mathcal{L}$
by
\begin{equation}
v_x\Omega(x,y):=\pi(q_y)v_x\;,
\end{equation}
where on the right we are identifying vectors in $\mathcal{L}_z$ with vectors in $\mathcal{H}$ projecting to $\pi(q_z)\in \textup{P}(\mathcal{H}),$ for $z\in M.$
\end{definition}
We have that \begin{equation}
\rho_x(q_yq_z)\pi(q_x)=\pi(q_x)\pi(q_y)\pi(q_z)\pi(q_x)=\Omega(x,z)\Omega(z,y)\Omega(y,x)\pi(q_x)\;,
\end{equation}
hence:
\begin{proposition}
 $\rho_x(q_yq_z)=\Omega(x,z)\Omega(z,y)\Omega(y,x).$
 \end{proposition}
Therefore, the 3–point functions $\Delta$ defined in both categories agree.
\begin{proposition}
$\Omega(x,x)=\mathds{1},$ $|\Omega(x,y)|<1$ if $x\ne y$ and $\Omega(x,y)=\Omega^*(y,x).$
\end{proposition}
\begin{proof}
That $\Omega(x,x)=\mathds{1}$ follows from the definition and that $|\Omega(x,y)|<1$ if $x\ne y$ follows from the Cauchy-Schwarz inequality: for normalized vectors $v,\,w,$ 
\begin{equation}
    \langle v,w\rangle\le 1
    \end{equation}
and equality holds if and only if $v,\,w$ define the same orthogonal projection. Since $q$ is injective, the result follows. That $\Omega(x,y)=\Omega^*(y,x)$ follows from $q_x^*=q_x.$
\end{proof}
\begin{proposition}
Condition 5 holds.  
\end{proposition}
\begin{proof}
This follows immediately from $|\Omega(x,y)|=\sqrt{\rho_x(q_y)}.$
\end{proof}
\begin{proposition}\label{rep}
$\int_M \Omega(x,y)\Omega(z,y)\,d\mu(z)=\Omega(x,y).$  
\end{proposition}
\begin{proof}
Since $v_x\Omega(x,z)\Omega(z,y)=\pi(q_y)\pi(q_z)v_x,$ this follows from
\begin{equation}
    \mathds{1}=\int_M q_z\,d\mu(z)\;.
\end{equation}
\end{proof}
\begin{proposition}
The map 
\begin{equation}
    \mathcal{H}\mapsto \mathcal{H}_{\textup{phy}}\;,\;\;v\mapsto \Psi\;,\; \Psi(x)=q_x v
\end{equation}
 is a unitary equivalence, with the inverse given weakly by 
 \begin{equation}\label{inte}
    \mathcal{H}_{\textup{phy}}\to\mathcal{H}\;,\;\;\Psi\mapsto \int_M \Psi(x)\,d\mu(x)\;.
\end{equation}
\end{proposition}
\begin{proof}
This follows from \cref{rep} and the resolution of the identity.
\end{proof}
Note that, \cref{inte} is defined on the entire Hilbert space of sections, but as a result of overcompleteness it isn't injective on the entire space. Sometimes one can use Lagrangian polarizations to determine an orthogonal basis of the physical Hilbert space.
\subsubsection{Pullback of Hermitian Form}
One can use $q$ to pullback the Hermitian form to a complex–valued form on $M.$ In most cases $q$ is a smooth embedding, therefore $M$ inherits a Riemannian metric. If in addition $q(M)\subset \textup{P}(\mathcal{H})$ is an almost complex submanifold then the pullback of the Hermitian form turns $M$ into a K\"{a}hler manifold. We can compute the pullback from $\Delta$:
\begin{definition}
Let $\alpha:G^{(n)}\to\mathbb{C}$ be smooth, normalized and invariant under even permutations with respect to $S_{n+1}\acts G^{(n)}.$ Define
\begin{equation}
\mathcal{J}_0(\alpha):\mathfrak{g}\otimes\cdots\otimes\mathfrak{g}\to\mathbb{C}\;,\;\;\mathcal{J}_0(\alpha)(X_1,\ldots,X_n)=X_1\cdots X_n\alpha(m,\cdot,\ldots,\cdot)\;,
\end{equation}
where for $1\le k\le n,$ $X_k\in \mathfrak{g}_m$ differentiates in the $(k+1)$th component of $\alpha.$
\end{definition}
\begin{proposition}
The pullback of the Hermitian form is given by
 \begin{equation}
   \mathcal{J}_0(\log{\Delta})\;.
     \end{equation}
\end{proposition}
\begin{remark}$\mathcal{J}_0(\alpha)\vert_{m}$ is naturally identified with the $n$-jet of $\alpha(m,\cdot,\ldots,\cdot).$ In addition, its antisymmetrization is equal to $\textup{VE}(\alpha)/2.$
\end{remark}
\section{Proof of Equivalence of Propagators and Path Integrals}\label{theproof}
In order to formally prove the desired equivalence of \cref{patho} with \cref{path1}, we state the following lemma, which uses a recent definition of the Riemann integral on manifolds (\cref{intee}):
\begin{lemma}\label{int}
Let $f:[0,1]\to\mathbb{R}$ be $C^1$ and let $F:[0,1]\times[0,1]\to\mathbb{R}$ be $C^1$ and such that it vanishes on the diagonal. Then
\begin{equation}\label{rss}
    \sum_{i=0}^{n-1}F(x_i,x_{i+1})\xrightarrow[]{\Delta x_i\to 0}\int_a^b f\,dx
\end{equation}
for all $0\le a\le b\le 1,$ where $a=x_0<\cdots<x_n=b,$ if and only if for all $x\in [0,1],$
\begin{equation}\label{condi}
    \partial_y F(x,y)\vert_{y=x}=f(x)\;.
\end{equation}
\end{lemma}
\begin{theorem}(formal)
$\Omega$ is equal to \cref{path1} if and only if it satisfies \cref{patho}.
\end{theorem}
\begin{proof}
Iterating (\ref{conditionf}) with $x_0=x,\,x_n=y,$ we have
\begin{equation}\label{prop}
\Omega(x,y)=\int_{M^{n-1}} \prod_{k=0}^{n-1} \Omega(x_k,x_{k+1})\,d\mu(x_1)\,\cdots\,d\mu(x_{n-1})\;.
  \end{equation}
Let $\gamma:[0,1]\to M$ be a $C^1$ path and let $0=t_0<\cdots<t_n=1.$ Then by \cref{int},
\begin{equation}\label{conv}
    \prod_{k=0}^{n-1} \Omega(\gamma_{t_k},\gamma_{t_{k+1}})\xrightarrow[]{\Delta t_i\to 0} P_{\Omega}(\gamma)\;,
\end{equation}
where $P_{\Omega}$ denotes parallel transport with respect to $\nabla_{\Omega}.$ Therefore, taking $n\to\infty$ in (\ref{prop}) gives
\begin{equation}
\Omega(x,y)=\int_{\gamma_0=x}^{\gamma_1=y}P_{\Omega}(\gamma)\,\mathcal{D}\gamma\;.
\end{equation}
We obtain \cref{path1} if and only if $\nabla_{\Omega}=\nabla.$    
\end{proof}
\subsection{The Measure $\mathbf{P(\gamma)\,\mathcal{D}\gamma}$}
$P(\gamma)\,\mathcal{D}\gamma$ can't be defined as a complex measure, but it can be defined as a formal inverse limit of an inverse system of complex measures. The reason the limit is formal is because the category of complex measure spaces doesn't have enough inverse limits.
\begin{definition}
A morphism between measurable spaces is a measurable function between them. A morphism between complex measure spaces is a measurable map $f:(M_1,\mathcal{F}_1,\mu_1)\to (M_2,\mathcal{F}_2,\mu_2)$ such that $f_*\mu_1=\mu_2.$ 
\end{definition}
\begin{definition}\label{invl}
Let $\{((M_i,\mathcal{F}_i,\mu_i),f_{ij})\}_{i,j\in J}$ be an inverse system in the category of complex measure spaces for which
\begin{equation}
   (M,\mathcal{F}):= \varprojlim \{((M_i,\mathcal{F}_i),f_{ij})\}_{i,j\in J}
\end{equation}
exists, and let 
\begin{equation}
    \pi_i:(M,\mathcal{F})\to (M_i,\mathcal{F}_i)
\end{equation}
be the projection map. We say that $g\in L^1(M,\mathcal{F},\mu)$ if there exists $i\in J$ and  $g_i\in L^1(M_i,\mathcal{F}_i,\mu_i)$ such that $g=g_i\circ\pi_i ,$ and we define 
\begin{equation}
    \int_M g\,d\mu:=\int_{M_i} g_i\circ\pi_i\,d\mu_i\;.
\end{equation}
\end{definition}
Note that, $\mu$ is just notation. It is formally the inverse limit of the measures $\mu_i,$ which may not actually exist. This definition is consistent because if $f:(M_1,\mathcal{F}_1,\mu_1)\to (M_2,\mathcal{F}_2,\mu_2)$ is a morphism and if $g\in L^1(M,\mathcal{F}_2,\mu_2),$ then
\begin{equation}
    \int_{M_1} g\circ f \,d\mu_1=\int_{M_2} g\,d\mu_2\;.
\end{equation}
\begin{definition}
Let $M$ be a manifold and let $\Delta_{[0,1]}$ denote the directed set of partitions of $[0,1],$ partially ordered by refinement. Let $x,y\in M$ and let $\mathcal{B}$ denote the Borel $\sigma$–algebra of 
\begin{equation}\label{borel}
   X=\{\gamma:[0,1]\to M:\gamma(0)=x,\,\gamma(1)=y\}\;.
\end{equation}
For $\Delta\in \Delta_{[0,1]},$ let $\mathcal{B}_{\Delta}$ be the finest sub–$\sigma$–algebra of $\mathcal{B}$ that doesn't separate paths that agree on $\Delta.$ That is, if $A\in\mathcal{B}$ then $A\in\mathcal{B}_{\Delta}$ if and only if: $\gamma_1\in A$ and $\gamma_1\vert_{\Delta}=\gamma_2\vert_{\Delta}$ implies that $\gamma_2\in A.$ 
\end{definition}
Note that, if $\Delta_1\le \Delta_2,$ then $\mathds{1}:(X,\Delta_2)\to (X,\Delta_1)$ is measurable. In the context of Poisson manifolds, it is more natural to take $\mathcal{B}_{\Delta}$ to be the finest sub–$\sigma$–algebra that doesn't separate paths that are homotopic relative to the points of $\Delta.$
\begin{proposition}\label{invm}
Let $\Omega$ be a propagator on $(M,d\mu)$ and let $\Delta\in\Delta_{[0,1]}$ be a partition with $(n+1)$ points. Consider the measurable space $(X,\mathcal{B}_{\Delta}).$ Let 
\begin{equation}
    \Omega_{\Delta}:M^{n-1}\to\mathbb{C}\;,\;\;\Omega_{\Delta}(x_1,\cdots, x_{n-1})=\Omega(x,x_1)\Omega(x_1,x_2)\cdots\Omega(x_{n-1},y)\;.
\end{equation}
Then $\{((X,\mathcal{B}_{\Delta},\Omega_{\Delta}\,d\mu),\mathds{1})\}_{\Delta\in\Delta_{[0,1]}}$ is an inverse system of  $\textup{Hom}(\mathcal{L}_x,\mathcal{L}_y)$–valued measure spaces\footnote{Because $\textup{Hom}(\mathcal{L}_x,\mathcal{L}_y)\cong \mathbb{C},$ only traditional measure theory is needed for this.} for which
\begin{equation}\label{inv}
  (X,\mathcal{B})=  \varprojlim(X,\mathcal{B}_{\Delta})\;.
\end{equation}
\end{proposition}
\begin{proof}
That \cref{inv} holds follows from the fact that $\mathcal{B}$ is generated by cylindrical sets, ie. sets of the form
\begin{equation}
    \{\gamma(t_1)\in B_1,\ldots,\gamma(t_n)\in B_n\}\;,
\end{equation}
where $B_1,\ldots,B_n$ are Borel measurable sets of $M.$
That this is an inverse system of  $\textup{Hom}(\mathcal{L}_x,\mathcal{L}_y)$–valued measure spaces follows directly from condition \ref{conditionf}.
\end{proof}
\begin{definition}
For $x,y\in M,$ we denote by $P(\gamma)\,\mathcal{D}\gamma$ the formal measure $\mu$ of \cref{invl}.
\end{definition}
The next proposition immediately follows from the definitions. 
\begin{proposition}
Let $t_1<\cdots<t_n\in (0,1)$ and $f_1,\ldots,f_n\in L^{\infty}(M).$ For $f:\{\gamma:[0,1]\to M\}\to \mathbb{C},$ let
\begin{equation}
  f_{t}(\gamma)=f(\gamma(t))\;.  
\end{equation}
Then\footnote{On the left, we are identifying an operator with its integral kernel.}
\begin{equation}\label{inde}
(Q_{f_1}\cdots Q_{f_n})(x,y)=\int_{\gamma(0)=x}^{\gamma(1)=y}f_{1,t_1}(\gamma)\cdots f_{n,t_n}(\gamma)\,P(\gamma)\,\mathcal{D}\gamma\;.
\end{equation}
\end{proposition}
Since \cref{inde} is independent of the values of $t_1,\ldots, t_n,$ it defines a linear map
\begin{equation}
    \bigoplus_{n=0}^{\infty}L^{\infty}(M)^{\otimes n}\to \textup{Hom}(\mathcal{L}_x,\mathcal{L}_y)\;.
\end{equation}
\section{Examples}
We'll first discuss the simplest example, given by overcomplete projective submanifolds, then we'll discuss $S^2$ from the perspective of Toeplitz quantization. We'll then survey examples coming from unitary representations of Lie groups and reproducing kernels.
\begin{remark}
Actually, the most basic example is the case of a compact manifold with a probability measure and a trivial flat connection. In this case, $\Omega$ is just the parallel transport map itself and the Hilbert space is one–dimensional (except the quantization map isn't injective). To get the example of a non–trivial flat connection one has to use the path integral of Poisson manifolds, since the fundamental groupoid needs to be used instead of the pair groupoid.
\end{remark}
\subsection{Overcomplete Projective Submanifolds}
Let $\mathcal{H}$ be a separable Hilbert space. The most natural quantizations are of overcomplete symplectic submanifolds of $\textup{P}(\mathcal{H}).$ On these spaces, the path integral is easily described and essentially no choices are required to do the quantization.
\begin{definition}\label{overcd}
We call a symplectic submanifold $M^{2n}\subset \textup{P}(\mathcal{H})$ overcomplete if there exists $C>0$ such that
\begin{equation}\label{overd}
    C\int_M q\,\omega_{\textup{FS}}^n
\end{equation}
is an orthogonal projection, where $q:M\subset \textup{P}(\mathcal{H})\xhookrightarrow{} \textup{B}(\mathcal{H})$ is the canonical embedding.
\end{definition}
If \cref{overd} is an orthogonal projection, then it is the identity on the subspace of $\mathcal{H}$ spanned by the rank–one subspaces associated to the points in $M.$ Therefore, without loss of generality we may assume that \cref{overd} is the identity. As a result, this notion of overcompleteness is consistent with that of \cref{maind} and the quantization is simply given by $q.$
\begin{remark}
Given any overcomplete projective manifold, we can generate another one by considering the time evolution under Schr\"{o}dinger's equation. In the case that $M=\mathbb{C}\textup{P}^n,$ this just results in Hamiltonian flow.
\end{remark}
\begin{proposition}\label{overcc}
Let $M\subset \textup{P}(\mathcal{H})$ be overcomplete and equipped with the pullback prequantum line bundle and Hermitian connection. Then there is a propagator $\Omega$ integrating the connection, with respect to $C\omega^n_{\textup{FS}}.$  
\end{proposition}
\begin{proof}
\begin{equation}\label{topp}
x\,\Omega([x],[y])=q_{[y]}x
\end{equation}
is a propagator integrating the connection, with respect to $C\omega_{\textup{FS}}^n.$
\end{proof}
To prove that \cref{topp} determines the right connection, it is either by definition or the following lemma is required:
\begin{lemma}\label{prophol}
Let $(\mathcal{L},\nabla)\to (M,I)$ be a holomorphic line bundle and let $\Omega\in \Gamma(\pi_0^*\mathcal{L}^*\otimes \pi_1^*\mathcal{L})$ be the identity along the diagonal and holomorphic with respect to $(-I,I).$ Then $\nabla_{\Omega}=\nabla.$
\end{lemma}
\begin{proof}
Since $\Omega$ is holomorphic, $\nabla_{(0,X)} \Omega=0$ for any $X\in T^{(0,1)}_x M.$ Furthermore, since $\Omega$ is constant along the diagonal, for any $X\in T M\otimes\mathbb{C}$
\begin{equation}
    0=d\Omega(X,X)=\nabla_{(X,0)}\Omega+\nabla_{(0,X)}\Omega\;.
\end{equation}
For $X\in T^{(1,0)}M,$ it follows that $\nabla_{(0,X)}\Omega=0\,.$ Therefore, $\nabla_{(0,X)}\Omega=0$ for all $X\in T_x M\otimes\mathbb{C},$ implying the result.
\end{proof}
In this case, the Hilbert space is $\mathcal{H}$ (or more accurately, it is the one space spanned by $M$) and the quantization map and adjoint are as follows:
\begin{equation}
    Q_f=C\int_M fq\,\omega^n_{\textup{FS}}\;,\;\;Q^{\dagger}_A([x])=\langle x,Ax\rangle\;,
\end{equation}
where $x\in [x]$ is any normalized vector. As for when a projective manifold is overcomplete, this is the case if the Bergman kernel is constant along the diagonal:
\begin{lemma}
Suppose $(M,\Omega,I)$ is a K\"{a}hler manifold with very ample prequantum line bundle $\mathcal{L}.$ If the Bergman kernel is constant along the diagonal, then the rank–one orthogonal projection
\begin{equation}
  \Psi\mapsto \Psi(x)\frac{B(x,\cdot)}{B(x,x)}
\end{equation}
symplectically embeds $M$ as an overcomplete submanifold of $P(L^2_{\textup{hol}}(M,\mathcal{L})).$ Furthermore, the pullback of the canonical line bundle with Hermitian connection agrees with the prequantum line bundle with Hermitian connection (\cite{cahen}).
\end{lemma}
\begin{proof}
This follows from the fact that the Bergman kernel has the overcompleteness property. The second part follows from \cref{prophol}.
\end{proof}
In \cite{cahen}, the Bergman kernel along the diagonal is denoted by $\theta.$ Those authors discuss consequences of it being constant. We've found that in addition, when it is constant the path integral associated to the prequantum connection exists exactly.
\\\\While the Bergman kernel isn't always constant, it is always asymptotically constant (\cite{ma}), as we will discuss in \cref{toep}. The idea of quantization then seems to be to determine an approximately sympletic, overcomplete embedding into projective space.
\\\\For the next proposition, see \cite{mar}, page 23 or \cite{cahen}, example 1.
\begin{proposition}
The Bergman kernel of a homogeneous K\"{a}hler manifold whose action lifts to the prequantum line bundle with Hermitian connection is constant along the diagonal.    
\end{proposition}
For example, the previous proposition applies to $\mathbb{C}\textup{P}^n, \textup{T}^*\mathbb{R}^n,$ Siegel upper half space.
\begin{remark}\label{witten}
The Bergman kernel is exactly of the form of equation (2.8) in \cite{brane}, and indeed it is equal to a delta function, but only on the physical Hilbert space. 
\\\\The symmetry problem raised by the authors only suggests that the path integral can't be defined uniquely. We circumvent this issue by defining the entire set of sections $\Omega$ which formally determine the path integral. 
\\\\The final point brought up by the authors is that the path integral doesn't have a useful perturbation theory. This can be observed in the following example. The perturbation series of $\Omega(x,y)$ in $\hbar$ is zero, except at $x=y$ where it's $1.$ However, after integrating $\Omega$ against smooth functions and sections we do get a non–zero perturbation series, eg. a formal deformation quantization.
\end{remark}
\subsubsection{The 2–Sphere}
In the following example we will compute the propagator of $S^2$ and the induced non-commutative product on the space of polynomial functions of degree $\le n.$  We will do this example from the perspective of Toeplitz quantization. This is a special case of overcomplete projective manifolds. First we'll compute $\Omega,$ then we'll compute the noncommutative product, then we'll show that for $n=1$ the map $Q^{\dagger}$ agrees with the usual identification of the Pauli matrices with the coordinate functions on $S^2\xhookrightarrow{}\mathbb{R}^3.$
\begin{exmp}
Let $M=S^2$ with symplectic form $n\omega,$ where $\omega$ is the Fubini–Study form. In the standard trivialization on the complement of the north pole,
\begin{equation}
    n\omega=ni\frac{dz\wedge d\bar{z}}{(1+|z|^2)^2}\;.
\end{equation}
\begin{enumerate}
    \item[$\mathbf{\Omega:}$] The Hermitian metric of the fiber of the prequantum line bundle $\mathcal{L}\to S^2$ over $z$ is given by 
\begin{equation}
    \langle \lambda,\beta\rangle_z=\frac{\bar{\lambda}\beta}{(1+|z|^2)^n}\;,
\end{equation}
and an orthonormal basis for holomorphic sections is given by 
\begin{equation}\label{basis}
\Psi_k(z)=\sqrt{\frac{n+1}{2\pi n}{n\choose k}}\,z^k\;,\;\;k=0,\ldots, n\;.
\end{equation}
The Bergman kernel is given by
\begin{align}
     B(w,z)=\sum_i \Psi_i^*(w)\otimes\Psi_i(z)=\frac{n+1}{2\pi n(1+|w|^2)^n}\sum_{i=0}^n {n\choose k}(\bar{w}z)^k=\frac{n+1}{2\pi n}\bigg(\frac{1+\bar{w}z}{1+|w|^2}\bigg)^n\;.
\end{align}
It is constant along the diagonal, so \cref{prophol} implies that the path integral is equal to
\begin{align}
    \Omega(w,z)=\frac{B(w,z)}{B(w,w)}=\bigg(\frac{1+\bar{w}z}{1+|w|^2}\bigg)^n\;.
    \end{align}
The projection $q_w$ is given by the projection onto the normalized section (which we'll also denote by $q_w$)
\begin{align}
    q_w(z)=\sqrt{\frac{n+1}{2\pi n}}\Omega(w,z)=
    \frac{1}{\sqrt{(1+|w|^2)^n}}\sum_{k=0}^n\sqrt{{n \choose k}}\bar{w}^k \Psi_k(z)\;.
\end{align}
\item[$\mathbf{Q^{\dagger}:}$] Using \cref{basis} to make the identification 
\begin{equation}
    B(\Gamma_{hol}(\mathcal{L}))\cong M_{n\times n}(\mathbb{C})\;,
    \end{equation} 
the embedding 
\begin{equation}\label{emb}
   Q^{\dagger}: B(\Gamma_{hol}(\mathcal{L}))\xhookrightarrow{}C^{\infty}(\mathbb{C}\textup{P}^1)\;,\;\; Q^{\dagger}_A(z)=\langle q_z,Aq_z\rangle\;.
\end{equation}
is given by
\begin{equation}
    Q^{\dagger}_A(z)=\frac{1}{(1+|z|^2)^{n}}\sum_{j,k=0}^n \sqrt{{n\choose j}{n\choose k}}z^j\bar{z}^k A_{jk}\;.
\end{equation}
The image of this map is 
\begin{equation}
   \frac{1}{(1+|z|^2)^n}P_n(z,\bar{z})\;,
\end{equation}
where $P_n(z,\bar{z})$ is the space of polynomials that are at most degree $n$ in each of $z,\bar{z}.$ This space of functions is equal to $P_n(x_1,x_2,x_3),$ the set of polynomials in $x_1, x_2, x_3$ that have degree $\le n,$ where $x_1,x_2,x_3$ are the coordinate functions on $\mathbb{C}\textup{P}^1\cong S^2\subset \mathbb{R}^3.$ The identification of these two spaces of functions is given by
\begin{equation}
x_1+ix_2=\frac{2z}{1+|z|^2}\;,\;\;x_3=\frac{-1+|z|^2}{1+|z|^2}\;.
\end{equation}
The noncommutative product inherited by $P_n(x_1,x_2,x_3)$ is such that
\begin{equation}
p_1 *_{1/n} p_2=p_1p_2+\frac{i}{2n}\big(\{p_1,p_2\}-ig^{-1}(X_{p_1},X_{p_2})\big)+\mathcal{O}(1/n^2)\;,
\end{equation}
where $g$ is the K\"{a}hler metric and $X_{p_1},X_{p_2}$ are the respective Hamiltonian vector fields. Explicitly, if we write such a function as 
\begin{equation}
\frac{a(z,\bar{z})}{(1+|z|^2)^n}\;,\;\;a(z,\bar{z})=\sum_{j,k=0}^{n}\sqrt{{n\choose j}{n\choose k}}a_{jk} z^j\bar{z}^k\;,
\end{equation}
then the product is given by
\begin{equation}
\frac{a(z,\bar{z})}{(1+|z|^2)^n}\ast_{1/n}\frac{b(z,\bar{z})}{(1+|z|^2)^n}=\frac{c(z,\bar{z})}{(1+|z|^2)^n}\;,\;\;c_{jk}=\sum_{i=0}^n a_{ji} \,b_{ik}\;.
\end{equation}
One can check that, for $\{i,j,k\}=\{1,2,3\},$ 
\begin{align}
& x_i\ast_{1/n} x_i=x_i^2+\frac{1}{n}(x_j^2+x_k^2)\;,
\\&(x_1+ix_2)\ast_{1/n}(x_1-ix_2)=x_1^2+x_2^2+\frac{1}{n}(1+x_3)^2\;,
\\& (x_1-ix_2)\ast_{1/n}(x_1+ix_2)=x_1^2+x_2^2+\frac{1}{n}(1-x_3)^2\;.
\end{align}
\item[$\mathbf{\sigma:}$] For $n=1,$ the map $A\mapsto Q^{\dagger}_A$ gives the standard identification of the Pauli matrices with the coordinate functions on $S^2$:
\begin{align}\label{pauli}
    \begin{pmatrix}
        0 & 0\\
        1 & 0
    \end{pmatrix}\mapsto x_1+ix_2\;,\;\;
    \begin{pmatrix}
        0 & 1\\
        0 & 0
    \end{pmatrix}\mapsto x_1-ix_2\;,\;\;
   \begin{pmatrix}
        -1 & 0\\
        0 & 1
    \end{pmatrix}\mapsto x_3\;.
\end{align}
These matrices are commonly denoted $\sigma_{+},\,\sigma_-,\,\sigma_z,$ respectively.
\\\\Regarding \cref{witten}, $\Omega(w,z)$ is of the form $f(w,z)^{1/\hbar}$ for $\hbar=1/n,$ where $f$ is such that $f(z,z)=1,\, |f(w,z)|<1$ for $w\ne z.$ Since 
\begin{equation}
    |x|<1\implies \frac{x^{1/\hbar}}{\hbar^k}\xrightarrow[]{\hbar\to 0}0 
\end{equation}
for all $k>0,$ it follows that the perturbation series of $\Omega$ vanishes for $w\ne z.$ 
\end{enumerate}
\end{exmp}
\subsection{Examples From Unitary Representations}
The following construction is dual to Kirillov's orbit method (see chapter 5 of \cite{klauder4}). Let $G$ be a compact Lie group (compactness can be relaxed). Let
\begin{equation}
    \pi:G\to U(\mathcal{H})
\end{equation}
be an irreducible unitary representation. Let $w\in \mathcal{H}$ be normalized and let $H\subset G$ be the subgroup of elements which act by scalar multiplication on $w.$ Let $M=G/H$ with $d\mu=dx$ the induced left invariant measure.  
\\\\For $x\in G/H,$ let $q_x$ be the orthogonal projection onto $U(x')w,$ where $x'\in G$ is any vector in the fiber over $x,$ ie.
\begin{equation}
    q_x v=\langle U(x')w, v\rangle\, U(x')w\;.
\end{equation}
This defines an injective map
\begin{equation}
    q:G/H\to \textup{P}(\mathcal{H})\;.
\end{equation} 
The overcompleteness axiom holds due to Schur's lemma: let $g\in G,\,v\in\mathcal{H}.$  Then
\begin{align}
    & U(g)\int_{G/H}q_x \,dx\;v=\int_{G/H}U(g)q_x v\,dx=\int_{G/H}\langle U(x')w, v\rangle\, U(g)U(x')w\,dx
    \\&\nonumber =\int_{G/H}\langle U(x')w, v\rangle\, U(gx')w\,dx=\int_{G/H}\langle U(g^{-1}x')w, v\rangle \,U(x')v_0\,dx
    \\&\nonumber =\int_{G/H}\langle U(x')w, U(g)v\rangle \,U(x')w\,dx
    \nonumber=\int_{G/H}q_x \,dx\; U(g)v\;,
\end{align}
where for the fourth equality we have used that $dx$ is left–invariant, and for the fifth equality we have used that $U$ is unitary. Therefore, 
\begin{equation}
    \int_{G/H}q_x \,dx
\end{equation}
commutes with $U(g)$ for all $g\in G,$ and it follows from Schur's lemma that it is constant. We can therefore rescale $dx$ so that this constant is $1.$
\subsection{Examples From Reproducing Kernels}
Let $(M,d\mu)$ be a manifold with a Borel measure, let $(\mathcal{L},\langle\cdot,\cdot\rangle)\to M$ be a line bundle with Hermitian connection and let $\mathcal{H}\subset L^2(M,\mathcal{L})$ be a closed subspace for which
\begin{equation}
    \mathcal{L}\to\mathcal{L}_x\;,\;\;\Psi\mapsto\Psi(x)
\end{equation}
is continuous, for all $x\in M$ (eg. if $\mathcal{H}$ is finite–dimensional). For each $x\in M$ we get a bounded sesquilinear form given by
\begin{equation}
\mathcal{H}\otimes_{\mathbb{C}}\mathcal{H}\to\mathbb{C}\;,\;\;(\Psi_1,\Psi_2)\mapsto \langle\Psi_1(x),\Psi_2(x)\rangle _x\;,
\end{equation}
and the Riesz representation theorem guarantees that there is a bounded operator $q_x$ such that 
\begin{equation}
    \langle\Psi_1(x),\Psi_2(x)\rangle _x=\langle\Psi_1,q_x\Psi_2\rangle_{L^2}\;.
\end{equation}
We can rescale $q_x$ so that it is a rank–one orthogonal projection — to see this, choose a normalized vector $l_x\in\mathcal{L}_x.$ Writing $\Psi(x)=\lambda l_x$ determines a bounded linear functional $\mathcal{H}\to\mathbb{C},\,\Psi\mapsto\lambda.$ Letting $v_x$ be the vector determined by the Riesz representation theorem, we have that $q_x(v)=\langle v_x,v\rangle\,v_x.$ After rescaling $d\mu(x)$ in the inverse way as done to $q_x,$ it satisfies
\begin{equation}
\langle\Psi_1,\Psi_2\rangle_{L^2}=\int_M\langle\Psi_1,q_x\Psi_2\rangle_{L^2}\,d\mu\;,
\end{equation}
and therefore $x\mapsto q_x$ satisfies the overcompleteness axiom. The map $M\to \textup{P}(\mathcal{H})\,,\;x\mapsto q_x$ is injective if and only if the pointwise inner product of sections separates points of $M$. These can be used to quantize all compact symplectic manifolds (\cite{bor}).
\section{Path Integral Quantization of Symplectic Manifolds}
In this paper we are focusing on the non–perturbative aspects of quantization, but a complete quantization of a symplectic manifold $(M^{2n},\omega)$ includes a nice perturbation theory with respect to an $\hbar$–dependency. Using path integrals, we want an $\hbar$–family of propagators $\Omega_{\hbar}$ with respect to $d\mu_{\hbar}=\omega^n_{\hbar}\,,$ such that applying the functor from the category of path integrals to the category of abstract coherent state quantizations results in an abstract coherent state \textbf{deformation quantization}. Here,
\begin{equation}
    \hbar\omega_{\hbar}\xrightarrow[]{\hbar\to 0}C\omega
\end{equation}
for some $C>0.$ The coarsest requirement is that the cohomology class of $\hbar\log{[\Delta_{\hbar}]}$ is equal to $[\omega].$ More precisely, we need $\hbar \textup{\textup{VE}}(\log{\Delta_{\hbar}})\xrightarrow[]{\hbar\to 0}i\omega,$ or equivalently $\hbar F(\nabla_{\Omega})\xrightarrow[]{\hbar\to 0}i\omega.$\footnote{This means that we are not requiring that $\hbar F(\nabla_{\Omega})=i\omega$ exactly, which means that the embedding into projective space needn't be symplectic. In this precise sense, a quantization is an approximately sympelctic, overcomplete embedding into projective space. See \cref{afs}.}
\\\\We also need that for $x\ne y,$ $\Omega_{\hbar}(x,y)\xrightarrow[]{\hbar\to 0}0$ in such a way that
\begin{equation}
    \int_{M^2} f(y)g(z)\Omega_{\hbar}(x,y)\Omega_{\hbar}(y,z)\Omega_{\hbar}(z,x)\,\omega_{\hbar}^n(y)\omega_{\hbar}^n(z)\xrightarrow{\hbar\to 0} f(x)g(x)\;,
\end{equation}
and that the left side is smooth at $\hbar=0.$
\\\\In practice, propagators tend to come in such families and an important example is of Toeplitz quantization.
\subsection{ Berezin–Toeplitz Quantization}\label{toep}
Let $(M^{2n},\omega,I)$ be a prequantizable compact K\"{a}hler manifold with very ample prequantum line bundle $(\mathcal{L},\nabla,\langle\cdot,\cdot\rangle)\to M.$ Let $\{\Psi_{\alpha}\}_{\alpha}$ be an orthonormal basis for the space of holomorphic sections of $\mathcal{L}^k.$ The Bergman kernel $B_k$ (\cite{ma}) is a section of $\pi_0^*\mathcal{L}^{k*}\otimes\pi_1^*\mathcal{L}^k\to M\times M$ and is given by
\begin{equation}
    B_k(x,y)=\sum_{\alpha}\Psi_{\alpha}(x)^*\otimes \Psi_{\alpha}(y)\;.
\end{equation}
This is the integral kernel for the orthogonal projection onto holomorphic sections. We define a propagator with respect to the measure determined by the rescaled symplectic form
\begin{equation}
\omega_k(x)=B_k(x,x)^{1/n}\omega(x)\;,
\end{equation}
given by
\begin{equation}
    \Omega_k(x,y)=\frac{B_k(x,y)}{\sqrt{B_k(x,x)}\sqrt{B_k(y,y)}}\;.
\end{equation}
Since $B_k(x,x)^{1/n}/k \xrightarrow[]{k\to\infty} 1/\pi\,,$ 
it follows that
\begin{equation}\label{afss}
 \frac{\omega_k}{k}\to \frac{\omega}{\pi}\;.
\end{equation}
Furthermore, for any $m$ (\cite{zel})
\begin{equation}\label{afs}
   \big\|\frac{q_k^*\omega_{FS}}{k}-\omega\big\|_{C^m}=\mathcal{O}(1/k)\;.
\end{equation}
\Cref{afss}, \cref{afs} say that $q_k$ approximately symplectically embeds $M$ into projective space as an overcomplete submanifold, \cref{overcd}. The quantum operators are given by
\begin{align}
&\nonumber Q_f(x,y)=\int_{M}f(z)\Omega_k(x,z)\Omega_k(z,y)\,\omega^n_{1/k}(z)
\\&\label{norml}=\frac{1}{\sqrt{B_k(x,x)}\sqrt{B_k(y,y)}}\int_{M}f(z)B_k(x,z)B_k(z,y)\,\omega^n(z)\;.
\end{align}
These are the integral kernels of the normalized  Berezin–Toeplitz operators — the integral kernels of the usual  Berezin–Toeplitz operators act on the Hilbert space of holomorphic sections and are given by \cref{norml}, without the fraction on the outside of the integral. There is a unitary equivalence intertwining the  Berezin–Toeplitz operators with the operators \cref{norml}, given by
\begin{equation}
L^2_{\textup{hol}}(M,\mathcal{L}^k)\to\mathcal{H}_{\textup{phy}}\;,\;\;\Psi\mapsto\frac{\Psi}{\sqrt{B_k}}
\end{equation}
We have that\footnote{This holds from estimates in \cite{ma}): $B_k(x,x)$ is bounded away from $0$ in $(k,x)$ and $|B_k(x,z)|\xrightarrow[]{k\to\infty} 0$ uniformly in $z$ on the complement of any open set containing $x.$}
\begin{align}
\rho_{x}(Q_f)=Q_f(x,x)= &\int_{M}f(z)|\Omega_k(x,z)|^2\,\omega^n_{1/k}(z)
\\& \label{lim}=\frac{1}{B_k(x,x)}\int_{M}f(z)|B_k(x,z)|^2\,\omega^n(z)\xrightarrow[]{k\to\infty}f(x)\;,
\end{align}
where to take the limit in \cref{lim} we are assuming that $f$ is continuous. 
\\\\Finally, there is a unique star product $\star_{1/k}$ such that (\cite{mar})
\begin{align}
    \frac{1}{k^m}\|Q_fQ_g-Q_{f\star_{1/k}^m g}\|\to\xrightarrow{k\to\infty}0\;,
\end{align}
where $\star_{1/k}^m$ is the truncation of the star product above order $m.$ By \cref{trace}
\begin{equation}
\textup{dim}\,\mathcal{H}_{\textup{phy}}=\textup{Vol}_{\omega^n_{1/k}}(M)\;.
\end{equation}
\section{The Path Integral of Lie Algebroids}
We will define the path integral of Poisson manifolds —mirroring \cref{pman}, we begin with a definition of the path integral of Lie algebroids. Before doing this, we set things up.
\\\\
Let $\Pi_1(\mathfrak{g})\rightrightarrows M$ be the source simply connected Lie groupoid integrating $\mathfrak{g}\to M,$ let $d\mu$ be a continuously varying measure on the orbits of $\mathfrak{g}$ and let $(\mathcal{L},\nabla,\langle\cdot,\cdot\rangle)\to \Pi_1(\mathfrak{g})$ be a multiplicative line bundle with Hermitian connection (\cite{eli}). We want to define the formal path integral 
\begin{equation}\label{pathalgp}
    \Omega(g)=\int_{[\gamma]= g}P(\gamma)\,\mathcal{D}\gamma\;,
    \end{equation}
where the integral is over algebroid morphisms $T[0,1]\to \mathfrak{g}$ with homotopy class $g\in \Pi_1(M)$ and $P(\gamma)$ denotes parallel transport over $\gamma.$
\begin{definition}
We let $(G,d\mu)\rightrightarrows M$ denote a Lie groupoid with source and target maps $s,t$ and a continuously varying measure $d\mu$ along the orbits, ie. for $f\in  C_c(M),$
\begin{equation}
    O\mapsto\int_O f\,d\mu
\end{equation}
is continuous, where $O$ is a point in the space of orbits.
\end{definition}
Since the target map restricted to a source fiber surjects onto an orbit, we get the following:
\begin{definition}
$d\mu$ induces a Haar measure on $G,$ for which the source fibers of $G$ are equipped with the following $\sigma$–algebra: for each $x\in M,$ the $\sigma$–algebra of $s^{-1}(x)$ is the smallest one for which
\begin{equation}
    t_x:s^{-1}(x)\to M
\end{equation}
is measurable. If $G$ is source simply connected, we may instead take the $\sigma$–algebra of $s^{-1}(x)$ to be the smallest one for which
\begin{equation}
    \tilde{t}_x:s^{-1}(x)\to \widetilde{M}
\end{equation}
is measurable, where $\widetilde{M}$ is the universal cover of $M$ and $\tilde{t}_x$ is the lift of $t\vert_{s^{-1}(x)}:s^{-1}(x)\to M.$ We will also denote this Haar measure by $d\mu.$
\end{definition}
We denote the groupoid convolution by
\begin{equation}
\ast:\Gamma(\mathcal{L})\otimes_{\mathbb{C}}\Gamma(\mathcal{L})\to\Gamma(\mathcal{L})\;.
\end{equation}
\begin{definition}\label{involution}
There is an involution $^*:\Gamma(\mathcal{L})\to\Gamma(\mathcal{L})$
defined as follows: using the multiplicative structure of $\mathcal{L},$ for each $g\in G$ there is a natural map
\begin{equation}
    \mathcal{L}_{g}\otimes\mathcal{L}_{g^{-1}}\to\mathbb{C}
\end{equation}
since the pullback of $\mathcal{L}$ to the identity bisection is trivial. Combining this with the Hermitian metric, we get an adjoint map 
\begin{equation}
    ^*:\mathcal{L}_g\to\mathcal{L}_{g^-1}\;.
    \end{equation}
The involution is defined by $\Psi^*(g)=(\Psi(g^{-1}))^*.$
\end{definition}
\begin{definition}
Any $\Omega\in\Gamma(\mathcal{L})$ for which $\Omega\vert_M=\mathds{1}$ determines a map 
\begin{equation}
    \nabla_{\Omega}:\mathfrak{g}\to T_{M}\mathcal{L}\;,
\end{equation}
which is obtained by differentiating $\Omega$ along the source fibers at the identity bisection.
\end{definition} 
\subsection{Definition of the Propagator}
\begin{definition}
With the previous notation, let $\Omega\in\Gamma(\mathcal{L})$ be continuous and such that its restriction to each source fiber is measurable. We say that $\Omega$ is an (equal–time) propagator if
\begin{enumerate}
    \item $\Omega\vert_{M}=\mathds{1},$
    \item $|\Omega(g)|<1$ for $g\not\in \textup{Iso}(M)_0,$\footnote{$\textup{Iso}(M)_0$ is the connected component of the bundle of isotropy groups of $G.$}
    \item $\Omega^*=\Omega,$
    \item $\Omega\ast\Omega=\Omega,$
    \item $\sup_{x\in M} \int_{s^{-1}(x)}|\Omega(g)|\,d\mu(g)\,<\infty\,,$

\end{enumerate}
If $\Omega$ is smooth and $\nabla_{\Omega}=\nabla\vert_{\mathfrak{g}},$ then we say that $\Omega$ is a propagator integrating $\nabla.$
\end{definition}
Conditions 1, 3, 4 say that $\Omega$ is a normalized self–adjoint idempotent. 
\begin{definition}
We define a 2–cochain
\begin{equation}
\Delta:G^{(2)}\to\mathbb{C}\;,\;\; \Delta(g_1,g_2)=\Omega(g_1)\Omega(g_2)\Omega(g_2^{-1}g_1^{-1})
\end{equation}
which we call the 3–point function.
\end{definition}
In the following, $\textup{VE}$ is the van Est map (\cref{vanest}):
\begin{definition} 
If $\textup{VE}(\log{\Delta})=\Pi\in \Lambda^2\mathfrak{g}^*$ we say that $\Omega$ is a propagator integrating $\Pi\,.$
\end{definition}
In the previous definition, $\Pi$ is equal to the restriction of the curvature of $\nabla$ to the source fibers. When quantizing Poisson manifolds it is really $\Pi$ that we are interested in, rather than the specific connection. 
\\\\Similarly to the case of the path integral of manifolds, $\Delta$ determines a cohomology class of the local groupoid:
\begin{definition}\label{normd}
Let $U$ be a neighborhood of $M\xrightarrow[]{}G^{(2)}$ such that $\Delta\vert_U$ is nowhere zero. We define
\begin{equation}
   [\Delta]:U\to\mathbb{C}^*\;,\;\; [\Delta](g_1,g_2)=\frac{\Omega(g_1)\Omega(g_2)\Omega(g_2^{-1}g_1^{-1})}{|\Omega(g_1g_2)|^2}\;.
\end{equation}
$[\Delta]$ is a 2-cocycle on the local groupoid, valued in $\mathbb{C}^*.$ It follows that $\log{[\Delta]}$ is a 2-cocycle on the local groupoid, valued in $\mathbb{C}$ (where we choose the logarithm so that $\log{[\Delta]}\vert_M=0$), and assuming it is smooth this determines a 2-cocycle on the Lie algebroid.
\end{definition}
\begin{remark}
For Poisson manifolds $(M,\Pi),$ the first condition formally says that
\begin{equation}
    \int_{X:TD\to T^*M}\mathcal{D}X\,e^{\frac{i}{\hbar}\int_D X^*\Pi}=1\;,
\end{equation}
where the path integral is over Lie algebroid morphisms $X:TD\to T^*M$ such that $X(0)=x,$ where $0\in\partial D.$
\end{remark}
\subsection{Quantization of Lie Algebroids}
The quantization of a Lie algebroid is essentially a smoothly varying quantization of its orbits. Before defining it, we note that there is a measure on the arrows of $G$ induced by the Haar measure and the measure on $M.$ Since $\Gamma(\mathcal{L})$ acts on itself from the right, we get a $W^*$–algebra $W^*(G,\mathcal{L})$ by taking the weak–closure of $C_c(G,\mathcal{L})\xhookrightarrow{}B(L^2(G,\mathcal{L})).$ 
\begin{proposition}\label{norm}
    $\int_{s(g)=x}|\Omega(g)|^2\,d\mu(g)=1\;.$
\end{proposition}
\begin{proof}
    This follows from conditions 1, 3, 4.
\end{proof}
\begin{definition}
We define $\mathfrak{g}_{\hbar}$ to be the corner associated to $\Omega,$ ie.
\begin{equation}
\mathfrak{g}_{\hbar}=\Omega\ast W^*(G,\mathcal{L})\ast\Omega\;.
\end{equation}
\end{definition}
In other words, $\mathfrak{g}_{\hbar}$ consists of sections which are fixed by $\Omega$ on the right and on the left. 
\\\\The following is a general property of corners of $W^*$–algebras and follows from the definition:
\begin{proposition}
The canonical map \begin{equation}
    r:W^*(G,\mathcal{L})\to \mathfrak{g}_{\hbar}\;,\;\;r(A)= \Omega\ast A\ast\Omega
\end{equation} is $^*$–linear, continuous and fixes $\mathfrak{g}_{\hbar}.$ That is, $r$ is a $^*$–linear retraction. 
\end{proposition}
\begin{definition}
Let 
\begin{equation}
L^p(M):=\{f:M\to\mathbb{C}:\textup{ess} \sup{\|f\|_p}<\infty\}\;,
\end{equation}
where $\|f\|_p$ is the orbit–wise constant function given by the orbit–wise $L^p$–norm.
\end{definition}
\begin{definition}
The quantization map is given by
\begin{equation}
Q:\bigoplus_{1\le p\le\infty} L^p(M)\to \mathfrak{g}_{\hbar}\;,\;\;Q_f=r(t^*f)\;.
\end{equation}
Equivalently, $Q_f=r(s^*f).$
\end{definition}
We can quantize $x\in M$ by defining $q_x=Q_{\delta_x}$ and these do resolve the identity, but they are only continuous within the orbits. This next proposition is a simple computation.
\begin{proposition}\label{useful}
\begin{equation}
(Q_f\ast Q_h)(g)=\int_{g'g''g'''=g}f(t(g'))h(t(g''))\Omega(g')\Omega(g'')\Omega(g''')\,d\mu(g')d\mu(g'')\;.
\end{equation}  
\end{proposition}
We now define the state map:
\begin{definition}
Let 
\begin{equation}
\rho:\mathfrak{g}_{\hbar}\to L^{\infty}(M)\;,\;\;\rho_A(x)=A(x)\;,
    \end{equation}
where we are identifying $x\in M$ with its identity arrow.
\end{definition}
The next proposition immediately follows from the definitions.
\begin{proposition}
\begin{equation}
    \rho_{Q_f}(x)=\int_{s^{-1}(x)}f(t(g))|\Omega(g)|^2\,d\mu\;.
\end{equation}
\end{proposition}
The map $f\mapsto\rho_{Q_f}$ generalizes the Berezin transform.
\begin{definition}
We say that $\rho(M)$ has enough states if $\rho_A(x)=0$ for all $x\in M$ implies that $A=0.$
\end{definition}
In the case that $\rho(M)$ has enough states we get an embedding $\mathfrak{g}_{\hbar}\xhookrightarrow{}L^{\infty}(M).$
\begin{definition}
We let
\begin{equation}
\mathcal{H}_{S}:=\{A\in\mathfrak{g}_{\hbar}:\rho_{A^*\ast A}\in L^1(M)\}\;.
\end{equation}
\begin{definition}
We have inner products on $\mathfrak{g}_{\hbar},\,L^2(M)$ valued in 
\begin{equation}
    \{f\in L^{\infty}(M): f\textup{ is constant along orbits}\}\;,
    \end{equation}
given by
\begin{align}
    &\langle A,B\rangle_{\mathcal{H}_S}(x)=\int_{O_x}\rho_{A^*B}\,d\mu\;,
    \\&\langle f,g\rangle_{L^2}(x)=\int_{O_x} \bar{f}g\,d\mu\;,
\end{align}
where $O_x$ is the orbit containing $x.$ 
\end{definition}
\end{definition}
The next proposition can be taken to be the  defining feature of abstract coherent state quantizations of Lie algebroids with orbit–wise measures:
\begin{proposition}
$Q\vert_{L^2}(L^2(M))\subset \mathcal{H}_S$ and $\rho\vert_{\mathcal{H}_S}=Q^{\dagger}\vert_{L^2}\,,$ in the sense that for $f\in L^2(M),\, A\in\mathcal{H}_S,$
\begin{equation}
\langle A,Q_f\rangle_{\mathcal{H}_S}=\langle \rho_A,f\rangle_{L^2}\;.
\end{equation}
\end{proposition}
\begin{proof}
Let $O_x$ be the orbit of $x.$ Then
\begin{align}
  &\int_{y\in O_x}|\rho_{Q_{f}^*\ast Q_f}|(y)\,d\mu(y)=\int_{y\in O_x} |Q_{\bar{f}}\ast Q_f|(y)\,d\mu(y)
  \\&\le\int_{y\in O_x}\int_{g'g''g'''=y}\big|\bar{f}(t(g'))f(t(g''))\Omega(g')\Omega(g'')\Omega(g''')\big|\,d\mu(g')d\mu(g'')d\mu(y)
  \\&\le C\int_{s(g''),t(g'')\in O_x}|\bar{f}(s(g''))f(t(g''))\Omega(g'')|d\mu(g'')
  \\&\le C\|f\|^2_{L^2(M)}\;,
\end{align}
where to go from the second to third line we have used \cref{useful}, the triangle inequality and conditions 2 and 5. From the third to fourth line we have used  and H\"{o}lder's inequality and \cref{norm}. Note that, $C$ is independent of $x.$
The adjoint property follows from
\begin{align}
&\int_{y\in O_x}(A^*\ast\Omega\ast t^*f\ast\Omega)(y)\,d\mu(y)=\int_{y\in O_x}(\Omega\ast A^*\ast \Omega)(y)f(y)\,d\mu(y)
\\&=\int_{y\in O_x}\widebar{A(y)}f(y)\,d\mu(y)\;,
\end{align}
where we have used that $\Omega\ast A\ast\Omega=A$ by definition of $\mathfrak{g}_{\hbar}.$
\end{proof}
\begin{remark}
Let $0,\,1,\,\infty$ are cyclically ordered points on the boundary of a disk. For a Poisson manifold $(M,\Pi),$ we formally have that 
\begin{equation}
    \rho_{Q_f\ast Q_h}(x)=\int_{X:TD\to T^*M}\mathcal{D}X\,f(X(0))h(X(1))e^{\frac{i}{\hbar}\int_D X^*\Pi}\;,
\end{equation}
where the path integral is over algebroid morphisms such that $X(\infty)=x.$ According to \cite{bon}, this is equivalent to the Poisson sigma model description of Kontsevich's star product Writing it out exactly,
\begin{equation}
    \rho_{Q_f\ast Q_h}(x)=\int_{(g_1,g_2)\in G^{(2)},\, s(g_1)=x}f(t(g_1))h(t(g_2))\Delta(g_1,g_2)\,d\mu(g_1)d\mu(g_2)\;.
\end{equation}
\end{remark}
\section{Quantization of Poisson Structures}
One way of quantizing a Poisson manifold having a dense symplectic leaf is by finding a propagator on the symplectic groupoid over the symplectic leaf that extends to the entire symplectic groupoid. We show two examples of this here, before discussing Riemann surfaces.
\subsection{Quartic Zero}
Consider the Poisson structure on $S^2$ which in coordinates on the complement of the south pole is given by 
\begin{equation}\label{quartic}
    \Pi=\frac{1}{2i}|z|^4\partial_z\wedge\partial_{\bar{z}}\;.
\end{equation}
This Poisson structure is symplectic on the complement of the north pole, $z=0.$ 
\\\\
The symplectic groupoid of \cref{quartic} is given by $T^*S^2\rightrightarrows S^2.$ In coordinates $(z,\lambda)$ and for $1-\lambda|z|^2\bar{z}\ne 0,$ the source and target maps are given by 
\begin{equation}
s(z,\lambda)=z\,,\;\;t(z,\lambda)=\frac{z}{1-\lambda|z|^2\bar{z}}\;.
\end{equation}
For arrows with source and target $z=0,$ the composition is given by vector addition, ie. 
\begin{equation}
 (0,\lambda_1)\cdot(0,\lambda_2)=(0,\lambda_1+\lambda_2)\;,
\end{equation}
and the full subgroupoid over the complement of the north pole is the pair groupoid.
\\\\The multiplicative line bundle is trivial 
and the Haar measure is given by
\begin{equation}
    d\mu=\frac{i}{2}|z|^4d\lambda\wedge d\bar{\lambda}\;.
\end{equation}
The propagator is given by
\begin{equation}
    \Omega(z,\lambda)=\frac{1}{2\pi\hbar}e^{\frac{-|\lambda|^2|z|^4+\bar{\lambda}z-\lambda\bar{z}}{4\hbar}}\;.
\end{equation}
This can be described geometrically as follows: let $\alpha\in T^*S^2,$ then
\begin{equation}
\Omega(\alpha)=e^{\frac{\Pi(\alpha,I(\alpha))}{4\hbar}}P(t\mapsto t\alpha)
\end{equation}
where $I$ is the almost complex structure and $P(t\mapsto t\alpha)$ is the parallel transport map over the curve 
\begin{equation}
  t\mapsto t\alpha\;,\;\;  t\in [0,1]\;.
\end{equation}
The $\sigma$–algebra on a source fiber over the open symplectic leaf is the Borel $\sigma$–algebra, and the $\sigma$-algebra on the source fiber over the north pole is the trivial one consisting of the empty set and the entire source fiber, with the measure of this source fiber equaling $1.$ 
\\\\On the coordinate patch containing the south pole, the resulting formal deformation quantization is given by the Wick algebra:
\begin{equation}
    f\star_{\hbar}g=\textup{prod}\circ e^{\hbar  \frac{\partial}{\partial\bar{z}}\otimes \frac{\partial}{\partial z}}f\otimes g\;.
\end{equation}
\subsection{Podle\`{s} sphere}
The Poisson manifold known as the Podle\`{s} sphere is the Poisson structure on $S^2$ given by 
\begin{equation}
\Pi=\frac{1}{2i}|z|^2(1+|z|^2)\partial_z\wedge\partial_{\bar{z}}\;.
    \end{equation}
Its geometric quantization is described in \cite{bon}. The symplectic groupoid is $T^*S^2\rightrightarrows S^2.$ The full subgroupoid over the complement of the north pole is the pair groupoid, and in coordinates on the complement of the south pole the source and target maps are 
\begin{equation}
    s(z,\lambda)=z\,,\;\;t(z,\lambda)=\frac{z}{1-\bar{\lambda}(1+|z|^2)\bar{z}}\;.
\end{equation}
In these coordinates the multiplicative line bundle is trivial. To describe the propagator, let 
\begin{equation}
   B(x,y)=\sum_{n=0}^{\infty}c_n(\bar{x}y)^n e^{\frac{1}{2\hbar}(Li_2(-|x|^2)+Li_2(-|y|^2)}\;,
\end{equation}
where $L_2(t)=-\int_0^t \frac{\log(1-t')}{t'}\,dt'$ is the dilogarithm and
\begin{equation}
    c_n=\Big(2\pi\int_0^{\infty}\frac{t^n}{\sqrt{1+t}}e^{\frac{1}{\hbar}Li_2(-t)}\Big)^{-1}\;.
\end{equation}
On the complement of the north pole, this is the Bergman kernel. This orthonormal basis was computed in \cite{bon}. The propagator is given by 
\begin{equation}
    \Omega(x,y)=\frac{B(x,y)}{\sqrt{B(x,x)}\sqrt{B(y,y)}}\;.
\end{equation}
\subsection{Constant Curvature Surfaces}
The complete simply connected surfaces of constant curvature $\le 0$ are $\mathbb{C},\, \mathbb{H}$ (we have already discussed the positive curvature case). The propagators are given by 
\begin{enumerate}
    \item $\mathbb{C}:$
  \begin{equation}
      \Omega(z_1,z_2)=e^{-\frac{|z_1|^2+|z_2|^2-2z_1\widebar{z_2}}{2\hbar}}\;,\;\;\omega_{\hbar}=\frac{i}{4\pi \hbar }dz\wedge d\bar{z}\;.
  \end{equation}
   \item $\mathbb{H}=\{z\in\mathbb{C}:y>0\}:$ 
    \begin{equation}
        \Omega(z_1,z_2)=\bigg(2i\frac{\sqrt{\textup{Im}(z_1)\textup{Im}(z_2)}}{z_1-\bar{z_2}}\bigg)^{\frac{2}{\hbar}}\;,\;\;\omega_{\hbar}=\frac{i}{4\pi \hbar }\frac{dz\wedge d\bar{z}}{\textup{Im}(z)^2}\;.
    \end{equation}
\end{enumerate}
These propagators are the integral kernels of the orthogonal projections onto square–integrable holomorphic functions. A simple computation shows that $\textup{VE}(\log{\Delta})$ is proportional to $\omega_{\hbar}.$
A direct computation shows the following:
\begin{lemma}
For both $\mathbb{C}$ and $\mathbb{H},\,$ the $\mathbb{C}^*$–valued 2-cocycle on the pair groupoid
\begin{equation}
  [\Delta](x,y,z)=\frac{\Omega(x,y)\Omega(y,z)\Omega(z,x)}{|\Omega(x,z)|^2}  
\end{equation}
 is invariant under orientation–preserving isometries, as is the map $(x,y)\mapsto |\Omega(x,y)|^2.$
\end{lemma}
As a result, we get a universal description of quantizations of the quotients of $\mathbb{C}$ and $\mathbb{H}$ by any subgroup of orientation–preserving isometries, since $[\Delta]$ descends to the source simply connected groupoid. This includes all complete Riemann surfaces of non–positive constant curvature. Using $[\Delta]$ allows for the simplest description of the multiplicative line bundle and it arises from a change of trivialization. We explain this in the following lemma, but first we recall \cref{normd}:
\begin{equation}
  [\Delta](g_1,g_2)=\frac{\Omega(g_1)\Omega(g_2)\Omega(g_2^{-1}g_1^{-1})}{|\Omega(g_1g_2)|^2}  \;.
\end{equation}
Assuming $\Delta$ is nowhere vanishing, this is a 2-cocycle defined on all of $G^{(2)}.$
\begin{lemma}
Let $(G,d\mu)$ be a Lie groupoid with a measure along the orbits, and let $\Omega$ be a non–vanishing propagator on $\mathcal{L}\to G^{(1)}.$ Then $G^{(1)}\times \mathbb{C}\to G$ with multiplication given by
\begin{equation}
(g_1,a)\cdot(g_2,b)=(g_1g_2,ab[\Delta](g_1,g_2))\;,
\end{equation}
with the metric given by $\langle (g,a),(g,b)\rangle_g=\bar{a}b|\Omega(g)|^2$ and the propagator given by the constant $1$ is isomorphic to $(\mathcal{L},\Omega).$ 
\end{lemma}
\begin{proof}
The isomorphism is given by
\begin{equation}
(g,a)\mapsto (g,a\Omega(g))\;.
\end{equation}
\end{proof}
To explicitly see that $1$ is indeed a propagator,
\begin{align}
(1\ast 1)(g)=&\nonumber\int_{\{g_1g_2=g\}}[\Delta](g_1,g_2)\,d\mu=\int_{\{g_1g_2=g\}}\frac{\Omega(g_1)\Omega(g_2)\Omega(g^{-1})}{|\Omega(g)|^2}\,d\mu
\\&=\frac{\Omega(g)\Omega(g^{-1})}{|\Omega(g)|^2}=1\;.
\end{align}
For a subgroup $\Gamma$ of orientation preserving isometries, $[M/\Gamma]$ isn't necessarily a manifold. In the case that it isn't, $\Pi_1([M/\Gamma])$ is described by the double groupoid (\cite{mehta}) $\textup{Pair}\,\Gamma\ltimes \textup{Pair}\,M\rightrightarrows \Gamma\ltimes M.$
\begin{corollary}
Let $M=\mathbb{C}$ or $\mathbb{H}$ and let $\Gamma$ be a subgroup of orientation–preserving isometries. Then $(\mathcal{L},\Omega)$ descends to $\Pi_1([M/\Gamma])\;.$
\end{corollary}
If $M/\Gamma$ is prequantizable to $\mathcal{L}_0\to M/\Gamma$ then $\mathcal{L}\cong s^*\mathcal{L}_0^*\otimes t^*\mathcal{L}_0.$ It follows that there is a representation of $T(M/\Gamma)_{\hbar}$ on $L^2(M/\Gamma,\mathcal{L}_0).$
\begin{appendices}\label{app}
We'll review relevant definitions and results about groupoids. We'll start with the definition of the nerve of a groupoid and the van Est map given in \cite{Lackman2}. The definitions we use are different than the standard ones (\cite{lackman0}, \cite{weinstein1}), but they are equivalent and are more suitable for our purposes.
\\\\A basic but important example of a Lie groupoid is the pair groupoid:
\begin{definition}\label{pair}
The pair groupoid is denoted $\textup{Pair}\,M\rightrightarrows M.$ It is the unique Lie groupoid such that between any two points of $M$ is a unique arrow. 
\end{definition}
\begin{definition}
Given an integrable Lie algebroid $\mathfrak{g},$ its source simply connected groupoid is denoted $\Pi_1(\mathfrak{g}).$ In the case that $\mathfrak{g}=TM,$ we may also denote it by $\Pi_1(M).$
\end{definition}
The space of $n$-composable arrows of the pair groupoid is $\textup{Pair}^{(n)}M=M^{n+1}.$ More generally:
\begin{definition}
Let $G\rightrightarrows M$ be a Lie groupoid with source and target maps $s, t.$ We define the nerve of $G$ in degree $n$ to be
\begin{equation} 
    G^{(n)}=\underbrace{G\sideset{_s}{_{s}}{\mathop{\times}} G \sideset{_s}{_{s}}{\mathop{\times}} \cdots\sideset{_s}{_{s}}{\mathop{\times}} G}_{n \text{ times}}\,.
    \end{equation}
\end{definition}
Using this definition, the action of the symmetric group is easy to describe:
\begin{definition}
For $\sigma\in S_{n+1}$ and for $(g_1,\ldots,g_n)\in G^{(n)},$ we let
\begin{equation}
    \sigma\cdot(g_1,\ldots,g_n):=(g^{-1}_{\sigma^{-1}(0)}g_{\sigma(1)},\ldots,g^{-1}_{\sigma^{-1}(0)}g_{\sigma(n)})\,,
\end{equation}
where $g_0:=\textup{id}({s(g_1)}).$ 
\end{definition}
The result is a point in $G^{(n)}$ whose common source is $t(g_{\sigma^{-1}(0)}).$ If $\sigma$ fixes $0\in\{0,\ldots,n\}$ then $\sigma$ is just a permutation. 
\begin{definition}
Let $G\rightrightarrows M$ be a Lie groupoid and $\mathfrak{g}\to M$ its Lie algebroid. We define $n$-cochains as follows:
\begin{align}
    &C^n(G)=\{\Omega:G^{(n)}\to\mathbb{C}\}\;,
    \\&C^n(\mathfrak{g})=\{\omega:\mathfrak{g}^{\otimes n}\to\mathbb{C}\}\;.
\end{align}
\end{definition}
For the most part, we are interested in $n$-cochains that are invariant under $A_n\subset S_n.$
\begin{definition}
Let $G\rightrightarrows M$ be a Lie groupoid. We define $\mathcal{A}^n_0 G$ to be those $n$–cochains on $G$ that are invariant under $A_n\subset S_n$ and that vanish on the identity bisection.
\end{definition}
A vector $\xi\in\mathfrak{g}$ at a point $x\in M$ is a vector tangent to the source fiber of $G$ at $M.$ Given an $n$-cochain $\Omega$ and a point $x\in X,$ we can restrict 
\begin{equation}
    \Omega:\underbrace{G\sideset{_s}{_{s}}{\mathop{\times}} \cdots\sideset{_s}{_{s}}{\mathop{\times}} G}_{n \text{ times}}\to\mathbb{C}
\end{equation}to a map
\begin{equation}
    \Omega_x:\underbrace{s^{-1}(x){\mathop{\times}}  \cdots{\mathop{\times}} s^{-1}(x)}_{n \text{ times}}\to\mathbb{C}\,,
\end{equation}and it makes sense to differentiate $\Omega_x$ in each of the $n$ components independently.
\begin{definition}\label{vanest}
Let $G\rightrightarrows M$ be a Lie groupoid and $\mathfrak{g}\to M$ its Lie algebroid. For each $n\ge 1$ we define the van Est map 
\begin{equation}
    \textup{VE}:\mathcal{A}^n_0 G\to \Lambda^n\mathfrak{g}\,,\;\; \Omega\mapsto \textup{VE}(\Omega)
\end{equation}
as follows: for $\xi_1,\ldots,\xi_n\in\mathfrak{g}_x,$ we let \begin{equation}
    \textup{VE}(\Omega)(\xi_1,\ldots,\xi_n)=n!\,\xi_n\cdots\xi_1\Omega_x\;,
    \end{equation}
where $\xi_i$ differentiates $\Omega_x$ in the $ith$ component.\footnote{The $n!$ is due to the fact that the standard definition of the van Est map involves an alternating sum over permutations, but it doesn't divide by the number of permutations.}
\end{definition}
This definition of the van Est map has the advantage that it is defined as a map $\mathcal{A}^n_0 G\to \Lambda^n\mathfrak{g}.$ Using the standard definition (\cite{weinstein1}), $C^{\infty}(M)$–linearity of $\textup{VE}(\Omega)$ is a property that needs to be checked. This simplifies the proof of \cref{intev}.
\\\\This next definition is the main component of the proof of the equivalence of categories in \cref{theproof}:
\begin{definition}\label{intee}
Let $M$ be an oriented $n$-dimensional manifold, let $\omega$ be an $n$-form on $M$ and let $\textup{VE}(\Omega)/n!=\omega.$ Then given a triangulation $\Delta_M$ of $M,$ the (generalized) Riemann sum of $\omega$ is defined to be
\begin{equation}
    \sum_{\Delta\in \Delta_M}\Omega(\Delta)\;,
\end{equation}
where the sum is over all $n$-dimensional simplices.
\end{definition}
\begin{theorem}\label{intev}
Suppose that $\textup{VE}(\Omega)/n!=\omega.$ Then
\begin{equation}
\sum_{\Delta \in\Delta_M }\Omega(\Delta)\xrightarrow[]{\Delta\to 0}\int_M\omega\;,
\end{equation}  
where the limit is taken over barycentric subdivisions of any triangulation $\Delta_M.$
\end{theorem}
\end{appendices}

\end{document}